\documentclass{amsart}

\allowdisplaybreaks[4]

\usepackage{cite}

\newcommand{\cross}{{\times}}
\newcommand{\dd}{\mathrm{d}}
\newcommand{\dint}{\displaystyle\int}

\newcommand{\RR}{\mathbb{R}}
\newcommand{\N}{\mathbb{N}}
\newcommand{\cO}{\mathcal{O}}
\newcommand{\cE}{\mathcal{E}}
\newcommand{\cH}{\mathcal{H}}
\newcommand{\cB}{\mathcal{B}}
\newcommand{\eps}{\varepsilon}

\newcommand{\mx}{\mathrm{max}}
\newcommand{\mn}{\mathrm{min}}
\newcommand{\vol}{\mathop{\mathrm{vol}}}
\newcommand{\meas}{\mathop{\mathrm{meas}}}

\newtheorem{theorem}{Theorem}[section]
\newtheorem{prop}[theorem]{Proposition}
\newtheorem{cor}[theorem]{Corollary}
\newtheorem{lemma}[theorem]{Lemma}

\theoremstyle{definition}
\newtheorem{rem}[theorem]{Remark}

\DeclareMathOperator{\dist}{dist}

\begin{document}

\title[Robin $p$-Laplacian]{On the $p$-Laplacian with Robin boundary conditions\\
and boundary trace theorems}

\author{Hynek Kova\v r\'\i k}

\address{DICATAM, Sezione di Matematica, Universit\`a degli studi di Brescia,
Via Branze 38, 25123 Brescia, Italy}
\email{hynek.kovarik@unibs.it}
\urladdr{http://dm.ing.unibs.it/~hynek.kovarik/}

\author{Konstantin Pankrashkin}

\address{Laboratoire de Math\'ematiques d'Orsay, Univ.~Paris-Sud, CNRS, Universit\'e Paris-Saclay, 91405 Orsay, France}
\email{konstantin.pankrashkin@math.u-psud.fr}
\urladdr{http://www.math.u-psud.fr/~pankrash/}

\date{\today}

\begin{abstract}
Let $\Omega\subset\RR^\nu$, $\nu\ge 2$, be a $C^{1,1}$ domain whose boundary $\partial\Omega$ is either compact
or behaves suitably at infinity.
For $p\in(1,\infty)$ and $\alpha>0$, define
\[
\Lambda(\Omega,p,\alpha):=\inf_{\substack{u\in W^{1,p}(\Omega)\\ u\not\equiv 0}}\dfrac{\displaystyle \int_\Omega |\nabla u|^p \mathrm{d} x - \alpha\displaystyle\int_{\partial\Omega} |u|^p\mathrm{d}\sigma}{\displaystyle\int_\Omega |u|^p\mathrm{d} x},
\]
where $\mathrm{d}\sigma$ is the surface measure on $\partial\Omega$.
We show the asymptotics
\[
\Lambda(\Omega,p,\alpha)=-(p-1)\alpha^{\frac{p}{p-1}} - (\nu-1)H_\mathrm{max}\, \alpha + o(\alpha),
\quad
\alpha\to+\infty,
\]
where $H_\mathrm{max}$ is the maximum mean curvature of $\partial\Omega$.
The asymptotic behavior of the associated minimizers is discussed as well.
The estimate
is then applied to the study of the best constant in a boundary trace theorem
for expanding domains, to the norm estimate for extension operators
and to related isoperimetric inequalities.
\end{abstract}

\maketitle

\section{\bf Introduction and main results}

\subsection{Problem setting}

Let $\Omega\subset \RR^\nu$, $\nu\ge 2$, be a domain with a sufficiently regular boundary.
For $\alpha>0$ and $p\in (1,\infty)$, consider the quantity
\begin{equation}   \label{eq-inf}
\Lambda(\Omega,p,\alpha):=\inf_{\substack{u\in W^{1,p}(\Omega)\\ u\not\equiv 0}}\dfrac{\displaystyle \int_\Omega |\nabla u|^p\,  \dd x - \alpha\displaystyle\int_{\partial\Omega} |u|^p\, \dd\sigma}{\displaystyle\int_\Omega |u|^p\, \dd x},
\end{equation}
where $\dd\sigma$ is the surface measure on $\partial\Omega$. Standard variational arguments show that under suitable assumptions, e.g.~if $\Omega$ is bounded
with a Lipschitz boundary, the problem \eqref{eq-inf} has a minimizer, see e.g. Proposition \ref{prop-minimizer} below. The respective Euler-Lagrange
equation takes the form of a non-linear eigenvalue problem
\begin{equation}
    \label{eq-pl}
-\Delta_p u= \Lambda\,  |u|^{p-2}u \text{ in } \Omega, \qquad
|\nabla u|^{p-2}\dfrac{\partial u}{\partial n}=\alpha |u|^{p-2} u \quad \text{ at }\quad \partial\Omega,
\end{equation}
where $\Delta_p$ is the $p$-Laplacian, $\Delta_p \,u = \nabla \cdot \big( |\nabla u|^{p-2}\nabla u \big)$, $n$ is the outer unit normal,
and $\Lambda= \Lambda(\Omega,p,\alpha)$. In the present paper
we work with $C^{1,1}$ domains, either bounded or with a suitable behavior at infinity (see below),
and we study the behavior of \eqref{eq-inf} as $\alpha$ tends to $+\infty$.
While the properties of $\Lambda(\Omega,p,\alpha)$ for $\alpha<0$ are well understood for any $p>1$, see e.g.~\cite{ekl} and 
references therein, the same problem for $\alpha\to +\infty$ was previously studied for the linear case $p=2$
only. It was shown in \cite{luzhu} that for bounded $C^1$ domains there holds
$\Lambda(\Omega,2,\alpha)=-\alpha^2+o(\alpha^2)$ as $\alpha\to +\infty$.
Under additional smoothness assumptions, a more detailed asymptotic expansion
\[
\Lambda(\Omega,2,\alpha)=-\alpha^2-(\nu-1)H_\mx(\Omega)\,\alpha+o(\alpha)
\]
was obtained first in \cite{p13,em} for $\nu=2$ and then in \cite{pp15,pp15b}
for the general case, where $H_\mx$ is the maximal mean curvature of the boundary.
Further terms in the asymptotic expansion can be calculated under
suitable geometric hypotheses, see e.g. \cite{dk,HK,HKR,pp15b}.
Non-smooth domains were studied as well, see e.g.~\cite{lp} and the recent preprint~\cite{bp}.
The above mentioned papers used a number of techniques which are specific
for the linear problems, such as the perturbation theory
for self-adjoint operators or a separation of variables, which are
not available for the $p$-Laplacian. 

\smallskip

In the present paper we are going to modify the existing approaches, which will allow us 
to feature the variational nature of the problem and to consider arbitrary values of $p$ in a unified way.
Furthermore, we work under weaker smoothness conditions
when compared to the preceding works, and only $C^{1,1}$ regularity
is assumed.

\subsection{Main result}

\noindent Let us pass to the exact formulation of our main result. In the sequel a domain $\Omega$ will be called {\bf admissible} if the following conditions are satisfied:
\begin{itemize}
\item[(i)] the boundary $\partial\Omega$ is $C^{1,1}$, i.e. is locally the graph of a function with a Lipschitz gradient,
\smallskip
\item[(ii)] the principal curvatures of $\partial\Omega$ are essentially bounded,
\item[(iii)] for some $\delta>0$ the map
\[
\partial\Omega\times(0,\delta)\ni(s,t)\mapsto s-tn(s)\in \big\{x\in \Omega:
\dist(x,\partial\Omega)<\delta\big\}
\]
is bijective.
\end{itemize}
The mean curvature $H$ of $\partial\Omega$ is the arithmetic mean of the principal curvatures,
and we set 
\begin{equation} \label{hmax}
H_\mx\equiv H_\mx(\Omega):=\mathop{\text{ess\,sup}} H. 
\end{equation}
We remark that we do not assume
that this value is attained.

An account of the differential geometry in the $C^{1,1}$ setting,
including the precise definition of the curvatures, can be found e.g. in~\cite[Section~3]{fu}.
In particular, the assumptions are satisfied for any domain
with a compact $C^{1,1}$ boundary. Another obvious example of an admissible domain is given by
any $C^{1,1}$ domain coinciding with a half-space outside a ball.
Our main result reads as follows:
\begin{theorem} \label{thm-main}
For any admissible domain $\Omega\subset\RR^\nu$ and any $p\in(1,\infty)$
there holds
\begin{equation} \label{eq-main}
\Lambda(\Omega,p,\alpha)=-(p-1)\alpha^{\frac{p}{p-1}} - (\nu-1)H_\mx(\Omega)\, \alpha
+ o(\alpha)
\text{ as } \alpha\to+\infty.
\end{equation}
\end{theorem}

\smallskip

\begin{rem}
Remark that the $C^{1,1}$ assumption is a minimal one to define
the curvature of the boundary. The asymptotics can be different for domains with a weaker regularity,
see e.g. Proposition~\ref{propa2} below for Lipschitz domains.

\end{rem}

\begin{rem} \label{rem-2}
A significant feature of our method of proving \eqref{eq-main} is that it does not assume existence of minimizers. This is important since there exist admissible domains for which problem \eqref{eq-inf} does not have a minimizer; for example if $p=2$ and $\Omega$ is an infinite cylinder with sufficiently smooth boundary, then it is easily seen that the infimum in \eqref{eq-inf} is not attained. 
\end{rem}

The proof of Theorem \ref{thm-main} is presented in Sections~\ref{sec2}--\ref{sec4}
and is organized as follows.
In Section~\ref{sec2} we estimate the eigenvalue $\Lambda(\Omega,p,\alpha)$ using some auxiliary operators
in a tubular neighborhood of $\partial\Omega$. In Section~\ref{sec3}
we obtain an upper bound by a suitable choice of test functions.
The lower bound is obtained in Section~\ref{sec4} using an analysis
of an auxiliary one-dimensional operator. 

In turns out that Theorem \ref{thm-main} has various applications to 
Sobolev boundary trace theorems, extension operators and isoperimetric inequalities. These are described 
in Section \ref{sec-appl}.

Apart from the asymptotic behavior of $\Lambda(\Omega,p,\alpha)$ it is natural to address
the question of the  behavior of the associated eigenfunction, in other words the minimizer of  \eqref{eq-inf}. 
Indeed, in \cite{dk} it was shown, for $p=2$, that as $\alpha\to \infty$, the eigenfunctions concentrate 
at the boundary of $\Omega$. In Section \ref{sec-eigenf} we carry this analysis, for a general $p>1$, a bit further. 
In particular, we prove an exponential localization
near the boundary, Theorem~\ref{prop-lip}, 
and a localization near the part of the boundary
at which the  mean curvature attains its maximum, Theorem~\ref{thm-locbd}.

Some explicitly solvable cases are discussed in Appendix~\ref{seca}.
In Appendix~\ref{secb} we prove an auxiliary elementary inequality used in the proofs, 
and in Appendix~\ref{appc} we show how the remainder estimate in Theorem~\ref{thm-main} can be improved 
under stronger regularity assumptions on $\partial\Omega$.

\section{\bf Applications} 
\label{sec-appl}

\subsection{\bf Best constants for boundary trace theorems}

The asymptotic expansion \eqref{eq-main} provides a number of consequences
for maps between various Sobolev spaces. Note first that a simple scaling argument gives
\begin{equation*}
\Lambda(\Omega,p,\mu^{p-1}\alpha)=\mu^{p} \Lambda(\mu \Omega,p,\alpha),
\quad
\mu>0.
\end{equation*}
In particular, as the half-space $\RR^{\nu-1}\cross\RR_+$
is invariant under dilations, the first
term on the right-hand side of  \eqref{eq-main} can be represented as
\begin{equation}
  \label{eq-half1}
\Lambda(\RR^{\nu-1}\cross\RR_+,p,\alpha)\equiv
-(p-1)\alpha^{\frac{p}{p-1}},  \quad \nu\ge 1,
\end{equation}
see also Appendix~\ref{seca}. Theorem~\ref{thm-main} thus admits the following version
for expanding domains:
\begin{cor} \label{cor1}
For any admissible domain $\Omega\subset\RR^\nu$, any $p\in(1,\infty)$ and
$\alpha>0$ one has, as $\mu$ tends to $+\infty$,
\[
\Lambda(\mu\Omega,p,\alpha)=\Lambda(\RR^{\nu-1}\cross\RR_+,p,\alpha) -(\nu-1)H_\mx(\Omega) \alpha \mu^{-1} +o(\mu^{-1}).
\]
\end{cor}
Furthermore, one checks easily that the function $\RR_+\ni\alpha\mapsto \Lambda(\Omega,p,\alpha)$ is strictly decreasing and continuous 
with $\Lambda(\Omega,p,0)=0$ and $\lim_{\alpha\to+\infty}\Lambda(\Omega,p,\alpha)=-\infty$ and, hence, defines a bijection
between $\RR_+$ and $\RR_-$. Denote by $S(\Omega,p,q)$
the best constant in the trace embedding $W^{1,p}(\Omega) \hookrightarrow L^q(\partial\Omega)$, which is defined through
\[
\dfrac{1}{S(\Omega,p,q)} = \sup_{\substack{u\in W^{1,p}(\Omega)\\ u\not\equiv 0}}
\dfrac{\left(\dint_{\partial\Omega} |u|^q\, \dd\sigma\right)^{\frac pq}}{\dint_\Omega\big( |\nabla u|^p +|u|^p\big)\, \dd x}, \qquad 
1<q<p_*,
\]
with
\[
p_*:=\begin{cases}
\dfrac{p\nu-p}{\nu-p}, & \quad \text{ if} \quad p\in(1,\nu),\\
\infty, & \quad \text{ if} \quad p\in[\nu,\infty).
\end{cases}
\]
The value of $S(\Omega,p,p)$ is then uniquely determined by the implicit equation
\begin{equation}
   \label{eq-spinv}
\Lambda\big(\Omega,p,S(\Omega,p,p)\big)=-1.
\end{equation}
In particular, in view of \eqref{eq-half1} it holds
\[
S(\RR^{\nu-1}\cross\RR_+,p,p)=(p-1)^{\frac{1-p}{p}}, \quad \nu\ge1,
\]
see also~\cite[Lemmas~3.1 \and 3.3]{efk}. 
Various estimates for $S(\Omega,p,q)$ were extensively studied in the literature,
see e.g. the review~\cite{rossi}. In particular,
it was shown in \cite{pino} that
for any $q>2$ there exists a constant $\gamma=\gamma(q,\nu)>0$ independent of $\Omega$
with
\begin{equation}
     \label{eq-pino}
S(\mu\Omega,2,q) = S(\RR^{\nu-1}\cross\RR_+,2,q)- \gamma\, H_\mx (\Omega)\, \mu^{-1} +o(\mu^{-1}) \text{ for $\mu\to +\infty$.}
\end{equation}
Furthermore, by \cite[Theorem~1.3]{bonder}, for each $\Omega$
there exist positive constants $c_1$
and $c_2$ such that for $\mu\to+\infty$ there holds
\[
c_1 \mu^\kappa\le S(\mu\Omega,p,q)\le c_2\mu^\kappa,
\qquad
\kappa:=\begin{cases}
\dfrac{(q-p)(\nu-1)}{q}, &  \ \text{ if} \quad  1<q<p,\\
0, &  \ \text{ if} \quad  p\le q< p_*.
\end{cases}
\]
In particular, for any $p\in(1,\infty)$ the constant $S(\mu\Omega,p,p)$ remains uniformly bounded
and separated from $0$ as $\mu\to+\infty$.
The substitution of~Corollary~\ref{cor1} into~\eqref{eq-spinv}
gives the following improvement in the spirit of~\eqref{eq-pino}:
\begin{cor} \label{cor2}
For any admissible domain $\Omega\subset\RR^\nu$ and any $p\in(1,\infty)$  there holds
\begin{equation} \label{eq-S}
S(\mu\Omega,p,p) = S(\RR^{\nu-1}\cross\RR_+,p,p) -\frac{(p-1)^{\frac{2-p}{p}}\, (\nu-1)}{p}\, H_\mx(\Omega)\, \mu^{-1} + o(\mu^{-1})
\end{equation}  
as $\mu$ tends to $+\infty$.
\end{cor}

\subsection{\bf Extension operators}
Recall that a bounded linear operator
$E$ from $W^{1,p}(\Omega)$ to $W^{1,p}(\RR^\nu)$
is called an extension operator if $E f$ coincides with $f$ in $\Omega$
for any $f$. The existence and various estimates
for the extension operators in terms of $\Omega$
are of interest, see e.g.~\cite{bur}.
We will be concerned with the lower bound for the norms
\[
\cE(\Omega,p)=\inf\big\{ \|E\|: \quad E: W^{1,p}(\Omega)\mapsto W^{1,p}(\RR^\nu) \text{ is an extension operator}\big\}.
\]
It is known, in particular, that
\begin{equation}
     \label{eq-ewp}
\cE(\Omega,p)
\ge \bigg(
1+\dfrac{S(\Omega^c,p,p)}{S(\Omega,p,p)}
\bigg)^{\frac{1}{p}}, \quad \Omega^c:=\RR^\nu\setminus {\overline \Omega}.
\end{equation}
see~\cite[Theorem~3.1]{lotor}. Note that the work~\cite{lotor} deals formally
with the case $p=2$ only, but the proof holds literally for any~$p\in(1,\infty)$.
Remark also that 
\[
\cE(\RR^{\nu-1}\cross\RR_+,p)=2^\frac{1}{p}.
\]
In fact, the lower bound follows from \eqref{eq-ewp}, and it is attained
by the operator of extension by parity
$(E f)(x_1,\dots,x_{\nu-1},x_\nu):=f\big(x_1,\dots,x_{\nu-1},|x_\nu|\big)$.
We have the following result:
\begin{cor}\label{cor3}
Assume that both $\Omega$ and $\Omega^c$ are admissible domains in
$\RR^\nu$ and that $p\in(1,\infty)$,
then for $\mu\to+\infty$ there holds
\[
\cE(\mu\Omega,p)\ge
\cE(\RR^{\nu-1}\cross\RR_+,p)
 +\dfrac{(p-1)^\frac{1}{p}}{ 2^\frac{p-1}{p}\, p^2}\,(\nu-1) \Big(
H_\mx(\Omega)+H_\mn(\Omega)\Big) \mu^{-1} +o(\mu^{-1}),
\]
where $H_\mn(\Omega):=\mathop{\mathrm{ess\,inf}} H$.
\end{cor}
\begin{proof}
We have
$H_\mx(\Omega^c)=-H_\mn(\Omega)$, and the substitution into~\eqref{eq-S}
and then into \eqref{eq-ewp} gives the result.
\end{proof}

\subsection{Isoperimetric inequalities}

Numerous works studied isoperimetric inequalities
for the quantities $\Lambda(\Omega,2,\alpha)$ and $S(\Omega,2,2)$.
In particular, in \cite{bar} it was conjectured that the balls maximize
$\Lambda(\Omega,2,\alpha)$ among all fixed volume domains for any $\alpha>0$.
An analogous question for $S(\Omega,2,2)$ was asked e.g. in \cite{rossi2}.
The conjecture was supported e.g. by the consideration
of the first and second variations of the respective functionals
and by showing
that the balls are at least local minimizers, see e.g. \cite{rossi2,fnt}.
It was shown only recently in \cite{fk15} that the conjecture in the general
form is wrong
by comparing the eigenvalues of the balls with those for the spherical
shells for large $\alpha$, while it remains true at least in two dimensions
for a restricted range of positive~$\alpha$. (It is worth noting that the case
$\alpha<0$ is well understood for any $p$, see \cite{dai,bucur}.)
In fact, the conjecture appears to be closely related
to some estimates for the maximum mean curvature $H_\mx$
as discussed in \cite{pp15}, and the asymptotics \eqref{eq-main}
and \eqref{eq-S} allow us to include into consideration
all possible values of $p$. More precisely, let us
recall the following known results:
\begin{itemize}
\item[(A)] The balls \emph{do not} minimize the quantity $H_\mx$ among the bounded $C^{1,1}$ domains
having the same volume.
In particular, consider the domains
\begin{gather*}
B_\rho:=\big\{x\in\RR^\nu: |x|<\rho\big\},\quad
U_{r,R}:=\big\{ x\in\RR^\nu: r<|x|<R\big\}, \\
0<r<\rho<R, \quad R^\nu-r^\nu=\rho^\nu, \nonumber\\
\text{then }
\vol B_\rho=\vol U_{r,R}, \quad H_\mx(B_\rho)=\dfrac{1}{\rho}>\dfrac{1}{R}=H_\mx(U_{r,R}). \nonumber
\end{gather*}
\item[(B)] For $\nu=2$, the balls are the strict minimizers of $H_\mx$
among all bounded \emph{simply connected} $C^2$ domains
of a fixed area, see e.g. the discussion in~\cite{kp-curv}.
\item[(C)] For $\nu\ge 3$, the balls are the strict minimizers of $H_\mx$
among the bounded \emph{star-shaped} $C^2$ domains of the same volume,
see e.g. \cite[Theorem~2]{pp15}.
\item[(D)]
For $\nu=3$, the balls \emph{do not} minimize 
the quantity $H_\mx$ among the bounded $C^{1,1}$ domains
having the same volume \emph{and homeomorphic to a ball}. Moreover,
there is no strictly positive lower bound for $H_\mx$ in terms of the volume.
The respective examples were constructed recently in \cite{fnt2}.

\end{itemize}

The combination of (A)--(C) with Theorem~\ref{thm-main} gives the following observations, with an arbitrary
$p\in(1,\infty)$:

\begin{itemize}

\item The balls do not maximize $\Lambda(\Omega,p,\alpha)$ among the domains of a fixed volume.
In particular, for sufficiently large $\alpha>0$ there holds
$\Lambda(B_\rho,p,\alpha)<\Lambda(U_{r,R},p,\alpha)$.

\item Let $B\subset \RR^\nu$ be a ball and $\Omega\subset \RR^\nu$ be a simply connected
bounded $C^2$ domain of the same volume, and for $\nu\ge 3$ assume additionally that $\Omega$
is star-shaped, then there exists $\alpha_\Omega>0$ such that
$\Lambda(B,p,\alpha)\ge\Lambda(\Omega,p,\alpha)$ for $\alpha>\alpha_\Omega$,
with an equality iff $\Omega$ is a ball.

\item
At least for $\nu=3$, the balls do not maximize $\Lambda(\Omega,p,\alpha)$  among the domains
homeomorphic to balls and having a fixed volume.

\end{itemize}

In a similar way, Corollary~\ref{cor2} combined with (A)--(C) gives the following
assertions valid for any $p\in(1,\infty)$:

\begin{itemize}

\item The balls do not maximize $S(\Omega,p,p)$ among the domains of a fixed volume.
In particular, for sufficiently large $\mu>0$ there holds
$S(\mu B_\rho,p,p)<S(\mu U_{r,R},p,p)$.

\item Let $B\subset \RR^\nu$ be a ball and $\Omega\subset \RR^\nu$ be a simply connected
bounded $C^2$ domain of the same volume, and for $\nu\ge 3$, assume additionally that $\Omega$
is star-shaped, then there exists $\mu_\Omega>0$ such that
$S(\mu B,p,p)\ge S(\mu \Omega,p,p)$ for $\mu>\mu_\Omega$,
with an equality iff $\Omega$ is a ball.

\item
At least for $\nu=3$, the balls do not maximize $S(\Omega,p,p)$ among the domains
homeomorphic to balls and having a fixed volume.

\end{itemize}
In view of \eqref{eq-pino}, the same considerations hold for $S(\Omega,2,q)$ with any $q>2$.

\section{\bf Proof of theorem~\ref{thm-main}: Bracketing and a change of variables}\label{sec2}

The construction of this section is quite standard
and represents a suitable adaptation of \cite[Sections 2.2 and 2.3]{pp15}.
For $s\in S:=\partial\Omega$, let $n(s)$ be outer unit normal and
$L_s:=\dd n(s): T_s S\to T_s S$ be the shape operator, which is defined for almost all $s\in S$, see \cite[Section~3]{fu}.
Recall that the eigenvalues $\kappa_1(s)\le\dots \le\kappa_{\nu-1}(s)$ of $L_s$
are the so-called principal curvatures at $s$, and the mean curvature $H(s)$ at $s$ is defined by
\[
H(s)=\dfrac{\kappa_1(s)+\dots+\kappa_{\nu-1}(s)}{\nu-1}\equiv \dfrac{1}{\nu-1}\, \mathop{\mathrm{tr}} L_s.
\]
By assumptions, all $\kappa_j$ are essentially bounded, and the maximal mean curvature
\[
H_\mx\equiv H_\mx(\Omega):=\mathop{\text{ess\,sup}} H
\]
is well-defined. In what follows, it will be convenient to use
the quantities
\[
M(s):=(\nu-1) H(s), \quad
M_\mx:=(\nu-1)H_\mx.
\]

For $\delta>0$ denote
\begin{gather}
      \label{eq-td}
\Omega_\delta:=\big\{
x\in \Omega: \, \inf_{s\in S} |x-s|<\delta
\}, \quad
\Theta_\delta:=\Omega\setminus \overline{\Omega_\delta},\\
\begin{aligned}\Lambda^N(\Omega,p,\alpha)&:=\inf_{\substack{u\in W^{1,p}(\Omega_\delta\cup \Theta_\delta)\\u\not\equiv 0}}
\dfrac{\displaystyle \int_{\Omega_\delta \cup \Theta_\delta} |\nabla u|^p \dd x - \alpha\displaystyle\int_S |u|^p\dd\sigma}{\displaystyle\int_{\Omega_\delta \cup \Theta_\delta} |u|^p\dd x},\\
\Lambda^D(\Omega,p,\alpha)&:=\inf_{\substack{u\in W^{1,p}(\Omega)\\u=0 \text{ on } \partial\Theta_\delta\\u\not\equiv 0}}
\dfrac{\displaystyle \int_{\Omega} |\nabla u|^p \dd x - \alpha\displaystyle\int_S |u|^p\dd\sigma}{\displaystyle\int_{\Omega} |u|^p\dd x},
\end{aligned}\nonumber
\end{gather}
then one clearly has
$ \Lambda^N(\Omega,p,\alpha)\le \Lambda(\Omega,p,\alpha)\le \Lambda^D(\Omega,p,\alpha)$.
Furthermore, denote
\begin{align*}
\Lambda^{N,\delta}(\Omega,p,\alpha)&:=\inf_{\substack{u\in  W^{1,p}(\Omega_\delta)\\u\not\equiv 0}}
\dfrac{\displaystyle \int_{\Omega_\delta} |\nabla u|^p \dd x - \alpha\displaystyle\int_S |u|^p\dd\sigma}{\displaystyle\int_{\Omega_\delta} |u|^p\dd x},\\
\widetilde W^{1,p}(\Omega_\delta)&:= \big\{u\in W^{1,p}(\Omega_\delta): \, u=0 \text{ on } \partial\Omega_\delta\setminus S\big\},\\
\Lambda^{D,\delta}(\Omega,p,\alpha)&:=\inf_{\substack{u\in \widetilde W^{1,p}(\Omega_\delta)\\ u\not\equiv 0}}
\dfrac{\displaystyle \int_{\Omega_\delta} |\nabla u|^p \dd x - \alpha\displaystyle\int_S |u|^p\dd\sigma}{\displaystyle\int_{\Omega_\delta} |u|^p\dd x}.
\end{align*}
It is easy to check that if $\Lambda^{D,\delta}(\Omega,p,\alpha)\le 0$,
then
$\Lambda^{N}(\Omega,p,\alpha)=\Lambda^{N,\delta}(\Omega,p,\alpha)$ and
$\Lambda^{D}(\Omega,p,\alpha)=\Lambda^{D,\delta}(\Omega,p,\alpha)$,
hence,
\begin{equation*}
\Lambda^{N,\delta}(\Omega,p,\alpha)\le \Lambda(\Omega,p,\alpha)
\le \Lambda^{D,\delta}(\Omega,p,\alpha).
\end{equation*}

We will study the quantities $\Lambda^{N,\delta}(\Omega,p,\alpha) $and
$\Lambda^{D,\delta}(\Omega,p,\alpha)$ using a change of variables.
By assumption we can choose $\delta>0$ sufficiently small such that
the map
\begin{equation}
        \label{eq-st}
\Sigma:=S\times(0,\delta)\ni(s,t)\mapsto \Phi(s,t):=s-tn(s)\in \Omega_\delta
\end{equation}
is bijective and uniformly locally bi-Lipschitz.
The metric $G$ on  $\Sigma$ induced by this embedding is
\begin{equation}
      \label{eq-gg}
G=g\circ (I_s-tL_s)^2 + \dd t^2,
\end{equation}
where $I_s:T_s S\to T_s S$ is the identity map,
and $g$ is the metric on $S$ induced by the embedding in $\RR^\nu$.
The associated volume form $\dd\Sigma$ on $\Sigma$ is
\begin{equation*}
\dd\Sigma(s,t) =\big|\det G(s,t)\big|^\frac{1}{2}\dd s\, \dd t=\varphi(s,t)\big|\det g(s)\big|^\frac{1}{2} \dd s \,\dd t=\varphi(s,t) \, \dd \sigma(s) \,\dd t,
\end{equation*}
where
$\dd \sigma(s)=\big|\det g(s)\big|^\frac{1}{2}\,\dd s$
is the induced $(\nu-1)$-dimensional volume form on $S$, and the weight $\varphi$ is given by
\begin{equation}
      \label{eq-r1}
\varphi(s,t):=\big|\det (I_s-t L_s)\big|= \prod_{j=1}^{\nu-1} \big(1-\kappa_j(s) t \big)=1-M(s) t +P(s,t)t^2,
\end{equation}
where $P$ is a polynomial in $t$ with coefficients which are essentially bounded functions of $s$,
and we assume in addition that $\delta>0$ is sufficiently small
to have $1/2\le \varphi\le 2$ almost everywhere in $\Sigma$.
In particular, 
\begin{align*}
\Lambda^{N,\delta}(\Omega,p,\alpha)&=\inf_{\substack{u\in  W^{1,p}(\Omega_\delta)\\u\not\equiv 0}}
\dfrac{\displaystyle \int_{\Sigma} \big|(\nabla u)\circ \Phi(s,t)\big|^p \dd \Sigma(s,t) - \alpha\displaystyle\int_S \big|u\circ \Phi(s,0)\big|^p\dd\sigma(s)}{\displaystyle\int_{\Sigma} \big|u\circ \Phi(s,t)\big|^p\dd \Sigma(s,t)},\\
\Lambda^{D,\delta}(\Omega,p,\alpha)&=\inf_{\substack{u\in \widetilde W^{1,p}(\Omega_\delta)\\u\not\equiv 0}}
\dfrac{\displaystyle \int_{\Sigma} \big|(\nabla u)\circ \Phi(s,t)\big|^p \dd \Sigma(s,t) - \alpha\displaystyle\int_S \big|u\circ \Phi(s,0)\big|^p\dd\sigma(s)}{\displaystyle\int_{\Sigma} \big|u\circ \Phi(s,t)\big|^p\dd \Sigma(s,t)},
\end{align*}
and the map $u\mapsto u\circ \Phi$ defines a bijection
between $W^{1,p}(\Omega_\delta)$ and $W^{1,p}(\Sigma)$
as well as between $\widetilde W^{1,p}(\Omega_\delta)$ and
$\widetilde W^{1,p}(\Sigma):=\big\{v\in W^{1,p}(\Sigma): \, v(\cdot,\delta)=0\big\}$.
Furthermore, for $v=u\circ \Phi$ we have
$\big| (\nabla u )\circ \Phi\big|^2
=G^{jk} \partial_j v \partial_k v$ with $(G^{jk}):=G^{-1}$,
and due to \eqref{eq-gg} we can estimate, with some $0<C_-<C_+$,
\[
C_- g^{-1} + \dd t^2\le G^{-1}\le C_+ g^{-1} + \dd t^2, \quad (g^{\rho\mu}):= g^{-1},
\]
which gives
$ C_- |\nabla_s v|^2 + |\partial_t v|^2
\le
\big| (\nabla u )\circ \Phi\big|^2
\le
C_+ |\nabla_s v|^2+ |\partial_t v|^2$ with
$|\nabla_s v|^2:=g^{\rho \mu} \partial_\rho v \partial_\mu v$.
Hence, with the notation 
\[
\Lambda^-(p,\alpha) :=\inf_{\substack{u\in W^{1,p}(\Sigma)\\u\not\equiv 0}}
\dfrac{\displaystyle \int_{\Sigma} \Big| C_- \big|\nabla_s u(s,t)\big|^2 + u_t(s,t)^2\Big|^{\frac p 2} \varphi(s,t) \dd\sigma(s)\dd t-\alpha \int_S \big|u(s,0)\big|^p\,\dd\sigma(s)}{\displaystyle\int_\Sigma \big|u(s,t)\big|^p \varphi(s,t)\dd\sigma(s)\dd t}
\]
and
\[
\Lambda^+(p,\alpha) := 
\inf_{\substack{u\in \widetilde W^{1,p}(\Sigma)\\u\not\equiv 0}}
\dfrac{\displaystyle \int_{\Sigma} \Big| C_+ \big|\nabla_s u(s,t)\big|^2 + u_t(s,t)^2\Big|^{\frac p 2} \varphi(s,t) \dd\sigma(s)\dd t-\alpha \int_S \big|u(s,0)\big|^p\,\dd\sigma(s)}{\displaystyle\int_\Sigma \big|u(s,t)\big|^p \varphi(s,t)\dd\sigma(s)\dd t}\, ,
\]
we conclude that 
\begin{equation} \label{DN}
\Lambda^-(p,\alpha)\ \le\  \Lambda(\Omega,p,\alpha)
\ \le\  \Lambda^+(p,\alpha)
\end{equation}
holds true provided $\Lambda^+(p,\alpha)\le 0$. Now we obtain separately
an upper bound for $\Lambda^+(p,\alpha)$, see Section~\ref{sec3},
and a lower bound for $\Lambda^-(p,\alpha)$, see Section~\ref{sec4}.

\section{\bf Proof of theorem~\ref{thm-main}: Upper bound}\label{sec3}

For an upper bound for $\Lambda^+(p,\alpha)$
we will test on functions of a special form.
To have shorter expressions we introduce the parameter
\begin{equation}
         \label{eq-beta}
\beta:=\alpha^{\frac{1}{p-1}}.
\end{equation}
Pick a $C^\infty$ function $\chi:(0,\delta)\to [0,1]$
which equals one in a neighborhood of $0$ and zero in a neighborhood of $\delta$, and define  
$\psi(t):=e^{-\beta t} \chi(t)$.
As $\beta$ tends to $+\infty$,
\begin{equation}
      \label{eq-phit}
\begin{aligned}
\int_0^\delta \psi(t)^p \varphi(s,t)\dd t&=\dfrac{1}{p\beta}-\dfrac{M(s)}{p^2\beta^2} + \cO \Big( \dfrac{1}{\beta^3}\Big),\\
\int_0^\delta \big|\psi'(t)\big|^p \varphi(s,t)\dd t&=\beta^p\bigg(\dfrac{1}{p\beta}-\dfrac{M(s)}{p^2\beta^2} + \cO \Big( \dfrac{1}{\beta^3}\Big)\bigg),
\end{aligned}
\end{equation}
where the remainder estimates are uniform in $s\in S$ due to the essential boundedness of the coefficients of $\varphi$.
Now we are going to consider two cases separately.

\subsection*{Case I: $p\in (1,2]$}
Using the inequality
$(a+b)^q\le a^q + b^q$ valid for $a,b\ge 0$ and $q\in (0,1]$
we estimate, with some $C>0$,
\begin{align*}
\Lambda^+(p,\alpha) & \le
\inf_{\substack{u\in \widetilde W^{1,p}(\Sigma)\\u\not\equiv 0}}
\bigg\{\displaystyle \int_{\Sigma} \Big( C\big|\nabla_s u(s,t)\big|^p + \big|u_t(s,t)\big|^p\Big) \varphi(s,t) \dd\sigma(s)\dd t
-\beta^{p-1} \int_S \big|u(s,0)\big|^p\dd\sigma(s)\bigg\}\\ &\qquad\qquad\quad \times \bigg\{
\displaystyle\int_\Sigma \big|u(s,t)\big|^p \varphi(s,t)\dd\sigma(s)\dd t\bigg\}^{-1}.
\end{align*}
Considering the functions $u$ of the form $u(s,t)=v(s)\psi(t)$ with $v\in W^{1,p}(S)$ and using
the estimates \eqref{eq-phit} we arrive, as $\beta\to+\infty$, at
\begin{multline}
    \label{eq-lplus}
\Lambda^+(p,\alpha)\le
\inf_{\substack{v\in W^{1,p}(S)\\v\not\equiv 0}} \Bigg\{
C \int_S \big|\nabla_s v(s)\big|^p \Big[ 
\dfrac{1}{p\beta}-\dfrac{M(s)}{p^2\beta^2} + \cO \Big( \dfrac{1}{\beta^3}\Big)
\Big]\dd\sigma(s)\\
+ \beta^p \int_S
\big|v(s)\big|^p \Big[ 
\dfrac{1}{p\beta}-\dfrac{M(s)}{p^2\beta^2} + \cO \Big( \dfrac{1}{\beta^3}\Big)
\Big]\dd\sigma(s)
-\beta^{p-1} \int_S \big|v(s)\big|^p \dd\sigma(s)
\Bigg\}\\
\times
\Bigg\{
\int_S
\big|v(s)\big|^p \Big[ 
\dfrac{1}{p\beta}-\dfrac{M(s)}{p^2\beta^2} + \cO \Big( \dfrac{1}{\beta^3}\Big)
\Big]\dd\sigma(s)
\Bigg\}^{-1},
\end{multline}
where the $\cO$-terms are uniform in $s\in S$ and do not depend on $v$.

To construct a suitable function $v$, we adapt the procedure appearing e.g. in \cite[Lemma~3.2]{bel} for Schr\"odinger operators with strong potentials.
Take an arbitrary $\varepsilon>0$. By assumption, the set
\[
S_\varepsilon:=\big\{s\in S: M_\mx-\varepsilon\le M(s)\le M_\mx\big\}
\]
has a non-zero measure, and almost any point $s$ of $S_\varepsilon$
has density one with respect to the Lebesgue measure, i.e. in our case $\sigma\big(\cB_\rho(s)\cap S_\varepsilon\big)/\sigma\big(\cB_\rho(s)\big)\to 1$
for $\rho\to 0$, where $\cB_\rho(s)$ is the geodesic ball in $S$ centered at $s$ of radius $\rho$,  see e.g. \cite[Section 1.7, Corollary 3]{ev}. Let us choose any $s\in S_\varepsilon$
with this property. In what follows, we denote by $B_r$ the ball of radius $r$ center at $0$ in $\RR^{\nu-1}$.
Let $y\in B_r$ be the Riemann normal coordinates centered at $s$, which will be used as local coordinates on $S$,
then for any $v\in W^{1,p}_0(B_r)$,
$v\not\equiv 0$, due to \eqref{eq-lplus} one has,
\begin{equation}   \label{eq-lpl2}
\Lambda^+(p,\alpha)\le
\dfrac{c\dint_{B_r}\big|\nabla v(y)\big|^p\dd \sigma(y)
+\beta^p \dint_{B_r} \big|v(y)\big|^p \Big[ (1-p) -\dfrac{M(y)}{p\beta}+\cO\Big(\dfrac{1}{\beta^2}\Big)\Big]
\dd\sigma(y)}{
\dint_{B_r} \big|v(y)\big|^p \Big[ 1 -\dfrac{M(y)}{p\beta}+\cO\Big(\dfrac{1}{\beta^2}\Big)\Big]
\dd\sigma(y)}.
\end{equation}
Assuming that $r$ is sufficiently small we have
$2^{-1}\dd y \le d\sigma(y) \le 2 \dd y$ in $B_r$.
Furthermore, due to the choice of $s$ we have
$\meas(B_\rho\cap S_\varepsilon)/\meas(B_\rho)\to 1$ as $\rho\to 0$, where $\meas$
stands for the Lebesgue measure in $\RR^{\nu-1}$.
Let $\mu\in (0,1/2)$, then  for sufficiently small
$\rho \in(0,r)$ one has $\meas(B_\rho\cap S_\varepsilon)\ge(1-\mu) \meas(B_\rho)$.
Denote 
$$
\theta\equiv \theta(\varepsilon,\rho):=
\meas(B_\rho\setminus S_\varepsilon)/\meas(B_\rho), 
$$
so that $0\le \theta\le \mu$. For a ball $\omega\subset \RR^{\nu-1}$, let $\Psi_\omega$ denote
a non-negative minimizer of
\begin{equation}
   \label{eq-pball}
\lambda_\omega:=
\inf \Big\{\int_\omega |\nabla f|^p \dd y: \, f\in W^{1,p}_0(\omega),  \ \int_\omega | f|^p \dd y= 1\Big\}.
\end{equation}
By~\cite{rad}, the function $\Psi_\omega$ is uniquely defined, and it is
radially decreasing. 
Furthermore,  set $\Phi_\omega=c_\omega \Psi_\omega$ with $c_\omega>0$
chosen such that
\[
\int_\omega \Phi_\omega(y)^p \dd \sigma(y)=1.
\]
Due to the above computations one has 
 $c_\omega\in (2^{-\frac{1}{p}},2^{\frac{1}{p}})$.
We are going to test in \eqref{eq-lpl2} on
\begin{equation}
    \label{eq-v}
v=\Phi_{B_\rho} \text{ extended by zero to} \ B_r. 
\end{equation}
Set $\eta:=\theta^{\frac{1}{\nu-1}}$, then $\meas (B_\rho\setminus S_\varepsilon)= \meas(B_{\eta\rho})$. As $\Phi_{B_\rho}$ is radially decreasing, we have
\[
\int_{B_\rho\setminus S_\varepsilon}
\Phi^p_{B_\rho} (y) \dd\sigma(y)\le
2
\int_{B_\rho\setminus S_\varepsilon}
\Phi^p_{B_\rho} (y) \dd y
\le
2\int_{B_{\eta\rho}}
\Phi^p_{B_\rho} (y) \dd y
\le
4 \int_{B_{\eta\rho}}
\Psi^p_{B_\rho} (y) \dd y.
\]
Furthermore, using the inequality $\eta\le \mu^{\frac{1}{\nu-1}}$
and the equality $\Psi_{B_\rho}(y)=\rho^{\frac{1-\nu}{p}} \Psi_{B_1}(y/\rho)$,
we have
\[
\int_{B_{\eta\rho}}
\Psi^p_{B_\rho} (y) \dd y
\le
\int_{B_{\mu^{\frac{1}{\nu-1}}\rho}}
\Psi^p_{B_\rho} (y) \dd y\\
=\rho^{1-\nu}\int_{B_{\mu^{\frac{1}{\nu-1}}\rho}}
\Psi^p_{B_1}(y/\rho)\dd y
=\int_{B_{\mu^{\frac{1}{\nu-1}}}}\Psi^p_{B_1}(y)\dd y
=:\gamma(\mu).
\]
Putting all together, for an aribitrarily chosen $\mu\in(0,1/2)$
we can take $\rho$ sufficiently small and make the choice \eqref{eq-v},
which gives
\[
\int_{B_\rho\setminus S_\varepsilon}
v(y)^p \dd y \le \gamma(\mu) = o(1), \qquad \mu\to 0. 
\]
Furthermore, using the fact that $\rho$ was chosen small but fixed
and that $p\le 2$, we have, with suitable $C_j>0$ and $\beta$ large enough
\begin{align*}
A &:=
c\dint_{B_r}\big|\nabla v(y)\big|^p\dd \sigma(y)
+\beta^p \dint_{B_r} v(y)^p \Big( (1-p) -\dfrac{M(y)}{p\beta}+\cO\big(\dfrac{1}{\beta^2}\big)\Big)
\dd\sigma(y)\\
&=c \lambda_{B_\rho}+(1-p)\beta^p
-\dfrac{M_\mx}{p} \beta^{p-1}
+ \dfrac{\beta^{p-1}}{p}\dint_{B_r} \big(M_\mx-M(y)\big)v(y)^p\dd\sigma(y)+C_1\beta^{p-2}\\
&=c \lambda_{B_\rho}+
(1-p)\beta^p -\dfrac{M_\mx}{p} \beta^{p-1}
+\dfrac{\beta^{p-1}}{p}\dint_{B_r\cap S_\varepsilon} \big(M_\mx-M(y)\big)
v(y)^p\dd\sigma(y)\\
&\qquad
+\dfrac{\beta^{p-1}}{p}\dint_{B_r\setminus S_\varepsilon} \big(M_\mx-M(y)\big)v(y)^p\dd\sigma(y)+C_1\beta^{p-2}\\
&\le
(1-p)\beta^p  -\dfrac{M_\mx}{p} \beta^{p-1}
+\dfrac{\varepsilon}{p}\beta^{p-1}\dint_{B_r\cap S_\varepsilon} v(y)^p\dd\sigma(y)\\
&\qquad{} +\dfrac{\|M_\mx-M\|_\infty}{p}\beta^{p-1}\!\!\dint_{B_r\setminus S_\varepsilon} v(y)^p\dd\sigma(y)
+C_2\\
&\le
(1-p)\beta^p  -\dfrac{M_\mx}{p} \beta^{p-1}
+\dfrac{\varepsilon}{p}\beta^{p-1}+
\dfrac{\|M_\mx-M\|_\infty \gamma(\mu)}{p}\beta^{p-1}
+C_2\\
&=(1-p)\beta^p - \dfrac{1}{p}\Big( M_\mx
-\varepsilon-
\|M_\mx-M\|_\infty \gamma(\mu)\Big)\beta^{p-1}
+C_2.
\end{align*}
In the same way we obtain
\begin{align*}
B&:=\dint_{B_r} \big|v(y)\big|^p \Big[ 1 -\dfrac{M(y)}{p\beta}+\cO\Big(\dfrac{1}{\beta^2}\Big)\Big]
\dd\sigma(y)\\
&=1 -\dfrac{M_\mx}{p} \beta^{-1}
+ \dfrac{1}{p} \beta^{-1}\dint_{B_r} \big(M_\mx-M(y)\big)v(y)^p\dd\sigma(y)+\cO(\beta^{-2})\\
&\le 1 -\dfrac{1}{p}\,\Big(M_\mx
-\varepsilon
-\|M_\mx-M\|_\infty \gamma(\mu)\Big)\beta^{-1}
+C_2 \beta^{-2}.
\end{align*}
For large $\beta$ one has $A<0$, and by \eqref{eq-lpl2} 
\begin{align*}
\Lambda^+(p,\alpha)\, & \le\, \dfrac{A}{B}\, \le\,  \dfrac{(1-p)\beta^p - \dfrac{1}{p}\Big( M_\mx
-\varepsilon-
\|M_\mx-M\|_\infty \gamma(\mu)\Big)\beta^{p-1}
+C_1}{1 -\dfrac{1}{p}\,\Big(M_\mx
-\varepsilon
-\|M_\mx-M\|_\infty \gamma(\mu)\Big)\beta^{-1}
+C_2 \beta^{-2}}\\
& \le
(1-p)\beta^p - \Big( M_\mx
-\varepsilon-
\|M_\mx-M\|_\infty \gamma(\mu)\Big)\beta^{p-1}
+C_3 \\
& =(1-p)\beta^p -M_\mx \beta^{p-1}
+\Big(\varepsilon+\|M_\mx-M\|_\infty \gamma(\mu)\Big)\beta^{p-1} + C_3,
\end{align*}
It follows that
\begin{align*}
\limsup_{\alpha\to+\infty} \dfrac{\Lambda(\Omega,p,\alpha) -(1-p)\alpha^{\frac{p}{p-1}}+M_\mx \alpha}{\alpha}
& \le
\limsup_{\alpha\to+\infty} \dfrac{\Lambda^+(p,\alpha) -(1-p)\alpha^{\frac{p}{p-1}}+M_\mx \alpha}{\alpha}\\
& =
\limsup_{\beta\to+\infty}
\dfrac{\Lambda^+(p, \alpha) -(1-p)\beta^p+M_\mx \beta^{p-1}}{\beta^{p-1}}\\
& \le \varepsilon+\|M_\mx-M\|_\infty \gamma(\mu).
\end{align*}
As $\varepsilon>0$ is arbitrary and $\gamma(\mu)$ can be made arbitrary small
by taking $\mu$ arbitrary small, we have
the sought estimate $\Lambda(\Omega,p,\alpha)\le (1-p)\alpha^{\frac{p}{p-1}}-M_\mx \alpha
+o(\alpha)$ for large $\alpha$.

\subsection*{Case II: $p\in (2,\infty)$}

Let $q>1$, then one can find a constant $c>0$ such that for all $\varepsilon_0\in (0,1)$
and all $a,b\ge 0$ there holds
\begin{equation}
         \label{eq-ineq}
(a+b)^q\le (1+\varepsilon_0)a^q + \dfrac{c}{\varepsilon_0^{q-1}}\,b^q,
\end{equation}
see e.g. Appendix~\ref{secb}.
Therefore, with a suitable $C>0$ and any $\varepsilon\in(0,1)$ one can estimate
\begin{align*}
\Lambda^+(p,\alpha)& \le
\inf_{\substack{u\in \widetilde W^{1,p}(\Sigma)\\u\not\equiv 0}}
\bigg\{\displaystyle \int_{\Sigma} \Big\{ C \varepsilon_0^\frac{2-p}{2}\big|\nabla_s u(s,t)\big|^p + (1+\varepsilon_0)\big|u_t(s,t)\big|^p\Big\} \varphi(s,t) \dd\sigma(s)\dd t\\
& \qquad\qquad\qquad -\beta^{p-1} \int_S \big|u(s,0)\big|^p\dd\sigma(s)\bigg\}\times
\bigg\{\displaystyle\int_\Sigma \big|u(s,t)\big|^p \varphi(s,t)\dd\sigma(s)\dd t\bigg\}^{-1}.
\end{align*}
Considering the functions $u$ of the form $u(s,t)=v(s)\phi(t)$ with $v\in W^{1,p}(S)$ and using
the estimates \eqref{eq-phit} we arrive, as $\beta\to+\infty$, at
\begin{multline}
  \label{eq-lam2}
\Lambda^+(p,\alpha)\le
\inf_{\substack{v\in W^{1,p}(S)\\v\not\equiv 0}} \Bigg\{
C \varepsilon_0^\frac{2-p}{2}\int_S \big|\nabla_s v(s)\big|^p \Big[ 
\dfrac{1}{p\beta}-\dfrac{M(s)}{p^2\beta^2} + \cO \Big( \dfrac{1}{\beta^3}\Big)
\Big]\dd\sigma(s)\\
+ (1+\varepsilon_0)\beta^p \int_S
\big|v(s)\big|^p \Big[ 
\dfrac{1}{p\beta}-\dfrac{M(s)}{p^2\beta^2} + \cO \Big( \dfrac{1}{\beta^3}\Big)
\Big]\dd\sigma(s)
-\beta^{p-1} \int_S \big|v(s)\big|^p \dd\sigma(s)
\Bigg\}\\
\times
\Bigg\{
\int_S
\big|v(s)\big|^p \Big[ 
\dfrac{1}{p\beta}-\dfrac{M(s)}{p^2\beta^2} + \cO \Big( \dfrac{1}{\beta^3}\Big)
\Big]\dd\sigma(s)
\Bigg\}^{-1},
\end{multline}
where the $\cO$-terms are uniform in $s\in S$ and do not depend on $v$ and $\varepsilon_0$,
and by taking $\varepsilon_0:=\beta^{-\frac{3}{2}}$ and choosing suitable $C_j>0$
we arrive at
\[
\Lambda^+(p,\alpha)
\le
\dfrac{C_1 \beta^{\frac{3(p-2)}{4}} \dint_S \big|\nabla_s v(s)\big|^p\,\dd\sigma(s)
+ \beta^p
 \dint_S \Big[(1-p) -\dfrac{M(s)}{p\beta} +\cO\Big(\dfrac{1}{\beta^{\frac{3}{2}}}\Big) \Big]\big|v(s)\big|^p\,\dd\sigma(s)
}{\dint_S
\big|v(s)\big|^p \Big[ 
1-\dfrac{M(s)}{p\beta} + \cO \Big( \dfrac{1}{\beta^2}\Big)
\Big]\dd\sigma(s)}.
\]
Now using the same notation and the same test function as in the case I we arrive at
\begin{align*}
\Lambda^+(p,\alpha)
& \le \dfrac{(1-p)\beta^p - \dfrac{1}{p}\Big( M_\mx
-\varepsilon-
\|M_\mx-M\|_\infty \gamma(\mu)\Big)\beta^{p-1}
+C_1\beta^{p-\frac{3}{2}}}{1 -\dfrac{1}{p}\,\Big(M_\mx
-\varepsilon
-\|M_\mx-M\|_\infty \gamma(\mu)\Big)\beta^{-1}
+C_2 \beta^{-2}}\\
& \le(1-p)\beta^p -M_\mx \beta^{p-1}
+\Big(\varepsilon+\|M_\mx-M\|_\infty \gamma(\mu)\Big)\beta^{p-1} + C_3 \beta^{p-\frac{3}{2}},
\end{align*}
where $C_j>0$ are suitable constants, and 
\[
\limsup_{\alpha\to+\infty} \dfrac{\Lambda(\Omega,p,\alpha) -(1-p)\alpha^{\frac{p}{p-1}}+M_\mx \alpha}{\alpha}
\le
\varepsilon+\|M_\mx-M\|_\infty \gamma(\mu),
\]
while $\varepsilon$ and $\gamma(\mu)$ can be chosen arbitrarily small, which gives the result.

\section{\bf Proof of theorem~\ref{thm-main}: Lower bound}\label{sec4}

\noindent

 The minoration $C_- \big|\nabla_s u(s,t)\big|^2 + u_t(s,t)^2\ge u_t(s,t)^2$
gives
\begin{equation}
      \label{eq-lest}
\Lambda^-(p,\alpha)\ge 
\inf_{\substack{u\in W^{1,p}(\Sigma)\\u\not\equiv 0}}
\dfrac{\displaystyle \int_S  \bigg( \int_0^\delta \big|u_t(s,t)\big|^p \varphi(s,t) \dd t -  \alpha \big|u(s,0)\big|^p \bigg) \dd\sigma(s)}{\displaystyle\int_\Sigma \big|u(s,t)\big|^p \varphi(s,t)\dd\sigma(s)\dd t}.
\end{equation}
Denote
\begin{equation}
    \label{eq-1d}
\lambda(\alpha,p,s):=
\inf_{u \in W^{1,p}(0,\delta)} \dfrac{\displaystyle \int_0^\delta \big|u'(t)\big|^p \varphi(s,t)\,\dd t - \alpha \big| u(0)\big|^p}{\displaystyle \int_0^\delta \big|u(t)\big|^p \varphi(s,t)\,\dd t}\, ,
\end{equation}
then for a.e. $s\in S$ one has
\[
\int_0^\delta \big|u_t(s,t)\big|^p \varphi(s,t) \dd t -  \alpha \big|u(s,0)\big|^p
\ge
\lambda(\alpha,p,s)
\int_0^\delta \big|u(s,t)\big|^p \varphi(s,t) \dd t,
\]
and \eqref{eq-lest} implies
\[
\Lambda^-(p,\alpha)\ge 
\inf_{\substack{u\in W^{1,p}(\Sigma)\\u\not\equiv 0}}
\dfrac{\displaystyle \int_\Sigma \lambda(\alpha,p,s)\big|u(s,t)\big|^p \varphi(s,t)\dd\sigma(s)\dd t}{\displaystyle\int_\Sigma \big|u(s,t)\big|^p \varphi(s,t)\dd\sigma(s)\dd t}
\ge
\mathop{\mathrm{ess\,inf}}_{s\in S} \lambda(\alpha,p,s).
\]
Hence the result follows from the following lemma:

\begin{lemma}\label{lemin}
There holds
\[
\lambda\equiv\lambda(\alpha,p,s)=(1-p)\alpha^{\frac{p}{p-1}}-M(s)\alpha+\cO(\alpha^{\frac{p-2}{p-1}}\log \alpha)
\quad \text{ as} \quad  \alpha\to +\infty,
\]
where the remainder estimate is uniform in $s$ outside a zero-measure set.
\end{lemma}

\begin{proof}
Introducing $\beta$ as in \eqref{eq-beta} and testing on $u(t)=e^{-\beta t}$ we obtain
by a direct computation the upper bound
\[
\lambda\le (1-p)\beta^p-M(s)\beta^{p-1}+\cO(\beta^{p-2})
\equiv
(1-p)\alpha^{\frac{p}{p-1}}-M(s)\alpha+\cO(\alpha^{\frac{p-2}{p-1}}),
\]
where the remainder depends on $\|\kappa_j\|_\infty$ only, see \eqref{eq-r1},
and, hence, is uniform in $s$ outside a zero-measure set.
In particular, for sufficiently large $\alpha$ we have,
\begin{equation}
    \label{eq-lb}
\lambda\le \dfrac{1-p}{2}\,\beta^p, 
\end{equation}
uniformly in $s\in S$. 
It follows by standard arguments that  problem \eqref{eq-1d}
admits a minimizer $v$, see e.g.~Proposition \ref{prop-minimizer} below. Without loss of generality we may assume that
\begin{equation}
    \label{eq-norm}
v\ge 0, \qquad \int_0^\delta v(t)^p \varphi(s,t)\,\dd t =  \|v\varphi^\frac{1}{p}\|_p^p=1.
\end{equation}
The Euler-Lagrange equation for $v$ reads
\begin{equation} 
  \label{eq-lagr}
\big( |v'|^{p-2}v' \varphi \big)'=-\lambda v^{p-1}\varphi,
\end{equation}
where the prime means the derivative in $t$,
with the boundary conditions
\begin{equation}
\label{eq-bceu}
\big|v'(0)\big|^{p-2} v'(0)=-\alpha v(0)^{p-1},
\quad
v'(\delta)=0. 
\end{equation}

In order to establish suitable decay properties of $v$ 
in the spirit of Agmon~\cite{agm}, let us
take $f\in C^1\big([0,\delta]\big)$ with $f\ge 0$ and $f(0)=0$.
Multiplying equation \eqref{eq-lagr} by $f^p v$, integrating  on $(0,\delta)$
by parts and  using the boundary conditions \eqref{eq-bceu} we arrive at
\begin{align}\label{eq-a1}
\lambda \int_0^\delta f(t)^p v(t)^p\varphi(s,t)\,\dd t & = -\int_0^\delta \big( |v'|^{p-2}v'\varphi\big)'(t) f(t)^p v(t)\,\dd t \nonumber \\
& =\int_0^\delta  \big|v'(t)\big|^{p-2}v'(t) (f^p v)'(t)\varphi(s,t)\,\dd t\nonumber\\
& =\int_0^\delta \big|v'(t)\big|^p f(t)^p \varphi(s,t)\, \dd t \nonumber\\
&\quad +p\int_0^\delta \big|v'(t)\big|^{p-2} v'(t) f(t)^{p-1}f'(t) v(t)\varphi(s,t)\,\dd t.             
\end{align}
An application of the Young inequality
\begin{equation} \label{eq-young}
|AB|\le \varepsilon |A|^q + \varepsilon^{-\frac{1}{q-1}}|B|^{\frac{q}{q-1}}, \quad
A,B\in\RR, \quad \varepsilon>0, \quad q >1,
\end{equation}
with
$A=\big|v'(t)\big|^{p-2} v'(t) f(t)^{p-1}$, $B=f'(t) v(t)$ and $q=p/(p-1)$
to the second term on the right-hand side of \eqref{eq-a1} gives
\begin{equation*}
-\lambda \int_0^\delta f(t)^p v(t)^p\varphi(s,t)\,\dd t
\le
(p\varepsilon -1)\int_0^\delta \big|v'(t)\big|^p f(t)^p\varphi(s,t)\, \dd t
+p\varepsilon^{-(p-1)} \int_0^\delta \big|f'(t)\big|^p v(t)^p\varphi(s,t)\,\dd t.
\end{equation*}
Taking $\varepsilon=1/(2p)$ and using \eqref{eq-lb} 
we arrive at
\begin{multline}
      \label{eq-ffp}
\int_0^\delta \big|v'(t)\big|^p f(t)^p\varphi(s,t)\, \dd t
+
(p-1)\beta^p\int_0^\delta v(t)^p f(t)^p\varphi(s,t)\, \dd t
\le
(2p)^p \int_0^\delta v(t)^p \big|f'(t)\big|^p\varphi(s,t)\, \dd t.
\end{multline}
Choose $f$ in the form $f(t)=\chi(t)e^{\omega t}$, where $\chi\in C^1\big([0,\delta])$ with
\[
\chi(0)=0, \quad  0\le \chi\le 1,\quad \chi(t)=1 \text{ for } t\ge \theta, \quad \|\chi'\|_\infty\le \dfrac{2}{\theta},
\]
and $\omega>0$ and $\theta\in (0,\delta)$ will be chosen later.
Using \eqref{eq-ineq} with $q=p$ and $\varepsilon=1$ we obtain
\begin{multline}
 \label{eq-fpr}
\int_0^\delta v(t)^p \big|f'(t)\big|^p\varphi(s,t)\, \dd t=
\int_0^\delta v(t)^p e^{p\omega t}\big|\chi'(t) +\omega \chi(t)\big|^p\varphi(s,t)\, \dd t\\
\begin{aligned}
&\le c\int_0^\delta v(t)^p \big|\chi'(t)\big|^p e^{p\omega t}\varphi(s,t)\dd t
+2 \omega^p \int_0^\delta v(t)^p \chi(t)^p e^{p\omega t}\varphi(s,t)\dd t\\
&\le c\int_0^\theta v(t)^p \big|\chi'(t)\big|^p e^{p\omega t}\varphi(s,t)\dd t
+2 \omega^p \int_0^\delta v(t)^p \chi(t)^p e^{p\omega t}\varphi(s,t)\dd t\\
&\le c_1 \theta^{-p}e^{p\omega\theta} \int_0^\theta v(t)^p \varphi(s,t)\dd t
+2 \omega^p \int_0^\delta v(t)^p \chi(t)^p e^{p\omega t}\varphi(s,t)\dd t\\
&\le c_1 \theta^{-p}e^{p\omega\theta} + 2 \omega^p \int_0^\delta v(t)^p \chi(t)^p e^{p\omega t}\varphi(s,t)\dd t,
\end{aligned}
\end{multline}
where $c_1:=2^p$, and on the last step we used the normalization \eqref{eq-norm} for $v$.
The substitution into \eqref{eq-ffp} gives
\begin{align*}
\int_0^\delta \big|v'(t)\big|^p f(t)^p\varphi(s,t)\, \dd t
+
(p-1)\beta^p\int_0^\delta v(t)^p f(t)^p\varphi(s,t)\, \dd t\, & \le \,
C_1 \theta^{-p}e^{p\omega\theta} \\
& \quad+C_2 \omega^p \int_0^\delta v(t)^p f(t)^p\varphi(s,t)\dd t
\end{align*}
with $C_1:=(2p)^p c_1$ and $C_2:=2(2p)^p$. Let us set
\[
\omega:=\kappa \beta  \quad \text{ with } \quad 
\kappa:=\Big(\dfrac{p-1}{2 C_2}\Big)^{\frac{1}{p}},
\]
so that
\begin{equation*}
\int_0^\delta \big|v'(t)\big|^p f(t)^p\varphi(s,t)\, \dd t
+
\dfrac{p-1}{2}\beta^p\int_0^\delta v(t)^p f(t)^p\varphi(s,t)\, \dd t\\
\le
C_1 \theta^{-p}e^{p\kappa \beta\theta}.
\end{equation*}
Finally, we put
\[
\theta:=\frac 1\beta\, .
\]
Then, with a suitable $C_2>0$,
\begin{equation*}
\int_0^\delta v(t)^p f(t)^p\varphi(s,t)\, \dd t\\
\le C_2, \quad
\int_0^\delta \big|v'(t)\big|^p f(t)^p\varphi(s,t)\, \dd t
\le C_2 \beta^p.
\end{equation*}
Further, as  $\varphi_0:=\inf_{(s,t)\in S\times(0, \delta)} \varphi(s,t)>0$, we have,
with $C_3:=C_2/\varphi_0>0$,
\begin{equation} \label{c3}
\int_0^\delta v(t)^p f(t)^p\, \dd t\le C_3, \qquad
\int_0^\delta \big|v'(t)\big|^p f(t)^p \, \dd t \le C_3 \beta^p,
\end{equation}
and \eqref{eq-fpr} gives
\[
\int_0^\delta v(t)^p \big|f'(t)\big|^p \, \dd t \le C_4\beta^p
\]
for some $C_4>0$. 
Using again \eqref{eq-ineq} with $q=p$ and $\varepsilon=1$ we conclude with
\[
\int_0^\delta \big|(v f )'(t)\big|^p \, \dd t
\le 2\int_0^\delta v(t)^p \big|f'(t)\big|^p \, \dd t
+
c \int_0^\delta \big|v'(t)\big|^p f(t)^p \, \dd t
\le
C_5 \beta^p.
\]
The integral bounds obtained allow us to estimate the values
of $v(\delta)$ and $v(0)$ as follows. First,
\[
v(\delta)^pf(\delta)^p= p\int_0^\delta (vf)'(t) (vf)^{p-1}(t)\,\dd t
\le p\|vf\|^{p-1}_p \big \| (vf)'\big\|_p \le C_6 \beta,
\]
implying
\begin{equation}
      \label{eq-a2}
v(\delta)^p\le C_6 \beta e^{-p\kappa \delta\beta}.
\end{equation}
Furthermore,
\begin{align*}
v(0)^p&=\varphi(s,0)v(0)^p\\
&= \varphi(s,\delta)v(\delta)^p
-p\int_0^\delta v(t)^{p-1} \varphi(s,t)^{\frac{p-1}{p}} v'(t) \varphi(s,t)^{\frac{1}{p}}\dd t
-\int_0^\delta v(t)^p \partial_t \varphi(s,t)\dd t\\
&\le \varphi(s,\delta)v(\delta)^p + p\|v \varphi^{\frac{1}{p}}\|^{p-1}_p \|v' \varphi^{\frac{1}{p}}\|_p
-\int_0^\delta v(t)^p \partial_t \varphi(s,t)\dd t.
\end{align*}
Using the normalization of $v$ and the estimate \eqref{eq-a2} we arrive at
\begin{equation}
       \label{eq-a3}
v(0)^p\le p \|v' \varphi^{\frac{1}{p}}\|_p -\int_0^\delta v(t)^p \partial_t \varphi(s,t)\dd t +\cO(\beta^{-2}).
\end{equation}
In order to estimate the integral on  the right-hand side we remark that, for any $b>0$ and as $\alpha$ is sufficiently large,
\begin{align*}
\int_{b\beta^{-1}\log \beta}^\delta  v(t)^p \varphi(s,t) \,\dd t
&\le f(b\beta^{-1}\log \beta)^{-p} \int_{c\beta^{-1}\log \beta}^\delta  v(t)^p f(t)^p\varphi(s,t) \,\dd t\\
&\le \dfrac{1}{\beta^{p\kappa b}} \int_{0}^\delta  v(t)^p f(t)^p\varphi(s,t) \,\dd t\le
\dfrac{C_3}{\beta^{p\kappa b}},
\end{align*}
where we have used \eqref{c3}. Hence for $b=2/(p\kappa )$ we obtain
\[
\int_{b\beta^{-1}\log \beta}^\delta  v(t)^p \varphi(s,t) \,\dd t=\cO(\beta^{-2})\, .
\]
On the other hand the normalization \eqref{eq-norm} implies
\[
\int_0^{b\beta^{-1}\log \beta}  v(t)^p \varphi(s,t) \,\dd t
=1 -\int_{b\beta^{-1}\log \beta}^\delta  v(t)^p \varphi(s,t) \,\dd t
=1+\cO(\beta^{-2}).
\]
Now, as $\partial_t\varphi/\varphi$ and its derivative in $t$ are uniformy bounded in $S\times(0,\delta)$,
we have
\begin{align*}
\int_0^\delta v(t)^p \partial_t \varphi(s,t)\dd t & =\int_0^\delta \dfrac{\partial_t \varphi(s,t)}{\varphi(s,t)} v(t)^p \varphi(s,t)\dd t
\\
&=\int_0^{b\beta^{-1}\log \beta} \dfrac{\partial_t \varphi(s,t)}{\varphi(s,t)} \,v(t)^p \varphi(s,t)\dd t+
\int_{b\beta^{-1}\log \beta}^\delta \dfrac{\partial_t \varphi(s,t)}{\varphi(s,t)} \,v(t)^p \varphi(s,t)\dd t\\
&=\int_0^{b\beta^{-1}\log \beta} \Big(\dfrac{\partial_t \varphi(s,0)}{\varphi(s,0)} + \cO(\beta^{-1}\log \beta)\Big) v(t)^p \varphi(s,t)\dd t + \cO(\beta^{-2})\\
&=
\int_0^{b\beta^{-1}\log \beta} \Big(-M(s) + \cO(\beta^{-1}\log \beta)\Big) v(t)^p \varphi(s,t)\dd t + \cO(\beta^{-2})\\
&=-M(s) + \cO(\beta^{-1}\log \beta),
\end{align*}
and the substitution into \eqref{eq-a3} gives
$v(0)^p\le p \|v' \varphi^{\frac{1}{p}}\|_p +M(s)+ \cO(\beta^{-1}\log \beta)$.
Finally, using the definition of $\lambda$ we infer that
\begin{multline*}
\lambda=\|v'\varphi^{\frac{1}{p}}\|^p_p-\alpha v(0)^p
\ge
\|v'\varphi^{\frac{1}{p}}\|^p_p-p\alpha \|v' \varphi^{\frac{1}{p}}\|_p -\alpha M(s)+ \cO(\alpha^{\frac{p-2}{p-1}}\log \alpha)\\
\ge
\inf_{x\in\RR_+} (x^p-p\alpha x) -\alpha M(s)+ \cO(\alpha^{\frac{p-2}{p-1}}\log \alpha)
=(1-p)\alpha^{\frac{p}{p-1}} -\alpha M(s)+ \cO(\alpha^{\frac{p-2}{p-1}}\log \alpha),
\end{multline*}
where the remainder estimate depends again on $\|\kappa_j\|_{\infty}$ only and is uniform
for $s$ outside a zero-measure set.
\end{proof}

\section{\bf Behaviour of minimizers: concentration effects} 
\label{sec-eigenf}

So far we have been dealing only with the asymptotic behavior of the eigenvalue $\Lambda(\Omega,p,\alpha)$.
In this section we will discuss some properties of the minimizers, as soon as they exist.
In contrast to the most of the paper, for a part of the results
we only require that $\partial\Omega$ be Lipschitz.
For the sake of completeness, we include the proof of the existence for bounded Lipschitz domains.

\begin{prop} \label{prop-minimizer}
If $\Omega\subset\RR^\nu$ is a bounded Lipschitz domain, then  the variational problem \eqref{eq-inf} has a minimizer for every $\alpha\in\RR$.
\end{prop}

\begin{proof}
Fix $\alpha\in\RR$ and let $\{u_j\}_{j\in\N}$ be a minimizing sequence for $\Lambda(\Omega,p,\alpha)$ normalized to one in $L^p(\Omega)$;
\begin{equation} \label{min-seq}
\lim_{j\to\infty} \Big(\int_{\Omega} |\nabla u_j|^p\, \dd x -\alpha \int_{\partial\Omega} |u_j|^p\, \dd\sigma\Big ) = \Lambda(\Omega,p,\alpha), \qquad 
\|u_j\|_{L^p(\Omega)} =1 \quad \forall\, j\in\N. 
\end{equation} 
By \cite[Thm.~ 1.5.1.10]{gr} for any $\eps\in (0,1)$ there exits a constant $K_\eps>0$ such that the upper bound 
\begin{equation} \label{grisvard}
\int_{\partial\Omega} |u|^p\, \dd\sigma \, \leq \, \eps\, \|\nabla u\|^p_{L^p(\Omega)} + K_\eps \, \|u\|^p_{L^p(\Omega)}
\end{equation} 
holds true for all $u\in W^{1,p}(\Omega)$. Applying this inequality with $u=u_j$ and $\eps$ sufficiently small, depending on $\alpha$, we deduce from \eqref{min-seq} that 
$$
\sup_{j\in\N}  \|\nabla u_j\|_{L^p(\Omega)} \, <\, \infty.
$$
It follows  that the sequence $\{u_j\}_{j\in\N}$ is bounded in $W^{1,p}(\Omega)$ and therefore admits a weakly converging subsequence, which we still denote by $\{u_j\}_{j\in\N}$. Let $u_\alpha$ be its weak limit in  $W^{1,p}(\Omega)$. The compactness of the embeddings $W^{1,p}(\Omega) \hookrightarrow L^p(\Omega)$ and $W^{1,p}(\Omega) \hookrightarrow L^p(\partial\Omega)$ implies that there exists another subsequence $\{u_{j_k}\}_{k\in\N}$ of $\{u_j\}_{j\in\N}$  such that 
$$ 
\|u_{j_k}\|_{L^p(\Omega)} \to  \|u_\alpha\|_{L^p(\Omega)} \quad \text{and} \quad \|u_{j_k}\|_{L^p(\partial\Omega)} \to  \|u_\alpha\|_{L^p(\partial\Omega)}
$$ 
as $k\to\infty$. Hence $\|u_\alpha\|_{L^p(\Omega)}=1$ and using the normalization of $u_j$ and the weak lower semi-continuity of $\int_\Omega |\nabla \cdot|^p$ in \eqref{min-seq} we arrive at 
\begin{align*}
\Lambda(\Omega,p,\alpha)& = \liminf_{k\to\infty} \Big(\int_{\Omega} |\nabla u_{j_k}|^p\, \dd x -\alpha \int_{\partial\Omega} |u_{j_k}|^p\, \dd\sigma\Big ) \geq \int_{\Omega} |\nabla u_\alpha|^p\, \dd x -\alpha \int_{\partial\Omega} |u_\alpha|^p\, \dd\sigma .
\end{align*}
This shows that $u_\alpha$ is a minimizer. 
\end{proof}

We mention the paper \cite{rossi3} discussing further properties of the minimizers such as the uniqueness and the strict positivity.
These properties are not used in our estimates below.

The following simple estimate for the eigenvalue is an adaption of a result from \cite{gs1}.

\begin{prop}\label{prop-lip1}
For any bounded Lipschitz domain $\Omega$ one has the inequality
\[
\Lambda(\Omega,p,\alpha)\le (1-p)\alpha^{\frac{p}{p-1}}
\]
for all $\alpha\ge 0$ and $p\in(1,\infty)$.
\end{prop}

\begin{proof}
Set $\beta:=\alpha^{\frac{1}{p-1}}$. Without loss of generality one may assume that $\Omega$ is contained
in the half-space $x_1>0$. Let us test on the function $u(x)=e^{-\beta x_1}$.
Consider the vector field $F(x)=(e^{-p\beta x_1},0,\dots,0)$, then the divergence theorem
gives
\begin{align*}
\int_{\partial \Omega} u^p\dd x_1
&=\int_{\partial\Omega}
e^{-p\beta x_1}\dd \sigma\ge \int_{\partial\Omega} F\cdot n \,\dd\sigma
=\int_\Omega \nabla \cdot F \, \dd x
= p\beta \int_\Omega e^{-p\beta x_1}\, \dd x
=p\beta \int_\Omega u^p \dd x,
\end{align*}
and
\begin{align*}
\Lambda(\Omega,p,\alpha) &\, \leq\,  \dfrac{\displaystyle\int_\Omega |\nabla u|^p\dd x - \beta^{p-1} \int_{\partial\Omega} u^p \dd\sigma}{\displaystyle\int_\Omega u^p\dd x}
\, \leq\, 
\dfrac{\beta^p \displaystyle\int_\Omega  u^p\dd x -p\beta^p \int_\Omega u^p \dd x}{\displaystyle\int_\Omega  u^p\dd x}
\\
 & =(1-p)\beta^p=(1-p)\alpha^{\frac{p}{p-1}}. \qedhere
\end{align*}
\end{proof}

Since the existence minimizers is not always guaranteed, see Remark \ref{rem-2}, 
in the following statements we will include it as an
assumption.
Similar to the proof of Lemma~\ref{lemin}, we obtain first
an exponential decay with respect
to the distance from the boundary using Agmon's approach~\cite{agm}.

\begin{theorem}\label{prop-lip}
Let $\Omega\subset\RR^\nu$ be a Lipschitz domain. Assume that for $\alpha$ large enough
the problem \eqref{eq-inf} admits a minimizer $u\equiv u_\alpha$, which we assume
non-negative and normalized by $\|u\|_{L^p(\Omega)}=1$, and that
\[
\Lambda\equiv \Lambda(\Omega,p,\alpha) \to-\infty, \qquad \alpha\to+\infty.
\]
Then for any $\tau \in (0,1)$ and any $a>0$ there holds,
with $\gamma:=(-\Lambda)^{\frac{1}{p}}$,
\begin{equation}
        \label{eq-decay}
\int\limits_{\dist(x,\partial\Omega)> \frac{a}{\gamma}}
\Big(\big|\nabla u(x)\big|^p -\Lambda \big|u(x)\big|^p\Big)
\exp\big(\tau \gamma \dist(x,\partial\Omega)\big)\dd x=\cO(\gamma^p)
\end{equation}
as $\alpha\to+\infty$. Furthermore, if $\Omega$ is a bounded Lipschitz domain or an admissible domain, then
for any $a>0$ there holds
\begin{align}
         \label{eq-decay2}
\int\limits_{\dist(x,\partial\Omega)> a \alpha^{-\frac{1}{p-1}}}\!\!\!\Big(\big|\nabla u(x)\big|^p + \alpha^{\frac{p}{p-1}}\big|u(x)\big|^p\Big)
\times 
\exp\Big((p-1)^{\frac{1}{p}}\alpha^{\frac{1}{p-1}}\dist(x,\partial\Omega)\Big)\dd x
& =\cO(\alpha^{\frac{2p}{p-1}}), 
\end{align}
as $\alpha\to+\infty.$
\end{theorem}

\begin{proof}
For $x\in \Omega$, denote $\rho(x):=\dist(x,\partial\Omega)$, then $|\nabla \rho|\le 1$.
Furthermore, for large $L>0$ denote $\rho_L(x):=\min\big\{\rho(x),L\big\}$, then we have again
$|\nabla \rho_L|\le 1$. The presence of the parameter $L$ is only relevant
for unbounded $\Omega$, as for a bounded domain one can take $L$ sufficiently large
to have $\rho_L=\rho$.

By standard arguments the minimizer $u$ satisfies \eqref{eq-pl}, which should be understood
in the weak sense, i.e.
\begin{equation}  \label{weak-el}
\Lambda(\Omega,p,\alpha) \int_\Omega |u|^{p-2}\, u\, \phi\, \dd x 
= 
\int_\Omega |\nabla u|^{p-2} \, \nabla u\cdot \nabla\phi \, \dd x
 -\alpha \int_{\partial\Omega} |u|^{p-2}\, u\, \phi \, \dd\sigma 
\end{equation}
holds for all $\phi \in W^{1,p}(\Omega) \cap L^\infty(\partial\Omega)$. 
The regularity theory of elliptic equations, see e.g.~\cite{to}, implies that $u$ is $C^{1,\epsilon}$ inside $\Omega$. 
Let $f$ be a non-negative bounded uniformly Lipschitz function defined in $\Omega$ and vanishing in a neighborhood
of $\partial\Omega$, then the equality \eqref{weak-el} with  $\phi:=f^p u$ and an integration by parts give  
\begin{align*}
\Lambda\int_\Omega u^p f^p \dd x=-\int_\Omega \nabla\cdot \big( |\nabla u|^{p-2}\nabla u\big) u f^p\dd x
&=\int_\Omega |\nabla u|^{p-2}\nabla u \cdot \nabla(f^p u)\,\dd x\\
&=\int_\Omega |\nabla u|^p f^p\dd x+p\int_\Omega |\nabla u|^{p-2}f^{p-1} u \nabla u \cdot \nabla f\dd x.
\end{align*}
Applying the Young inequality \eqref{eq-young} 
with
$A=\big| |\nabla u|^{p-2} \nabla u f^{p-1} \big| $, $B= \big| u \nabla f\big|$ and $q=p/(p-1)$
to the second term on the right-hand side we obtain, for any $\eps\in(0,1)$, 
\[
-\Lambda\int_\Omega u^p f^p \dd x
\le
(p\varepsilon-1)\int_\Omega |\nabla u|^p f^p\dd x+p\varepsilon^{1-p}\int_\Omega u^p |\nabla f|^p\dd x.
\]
Furthermore, let $\gamma$ and $\varepsilon_0\in[0,1)$
be such that $\gamma\to+\infty$ for $\alpha\to+\infty$
and
\begin{equation}
    \label{eq-eps0}
-\Lambda\ge (1-\varepsilon_0)^p\gamma^p \text{  for large $\alpha$}.
\end{equation}
In particular, one can simply take $\gamma:=(-\Lambda)^{\frac{1}{p}}$ and $\varepsilon_0\in(0,1)$.
We thus have 
\[
(1-\varepsilon_0)^p\gamma^p \int_\Omega u^p f^p \dd x
+(1-p\varepsilon)\int_\Omega |\nabla u|^p f^p\dd x\le p\varepsilon^{1-p}\int_\Omega u^p |\nabla f|^p\dd x.
\]
Furthermore, we may assume that $\delta_0:=1-p\varepsilon>0$, then
\begin{equation}
       \label{eq-f1}
(1-\varepsilon_0)^p\gamma^p\int_\Omega u^p f^p \dd x
+\delta_0\int_\Omega |\nabla u|^p f^p\dd x\le \dfrac{p^p}{(1-\delta_0)^{p-1}}\int_\Omega u^p |\nabla f|^p\dd x.
\end{equation}
To estimate the term on the right-hand side, let us take a function $f$ of a special form. Namely, we let
$a>0$ and $\chi\in C^\infty (\RR)$ be such that
\[
\chi:\RR\to [0,1], \quad \chi(t)=0 \text{ for $t$ close to $0$,}
\quad
\chi(t)=1 \text{ for } t\ge a,
\quad c_0:=\|\chi'\|_\infty,
\]
and set 
$$
f(x):=\chi\big( \gamma \rho(x)\big) e^{k \gamma \rho_L(x)}, 
$$
where the constant $k>0$ is to be chosen later. Hence
\[
\nabla f (x)= \gamma\Big(\chi'\big( \gamma \rho(x)\big) e^{k \gamma \rho_L(x)}\nabla \rho(x) +k \chi\big( \gamma \rho(x)\big) e^{k \gamma \rho_L(x)} \nabla \rho_L(x)\Big),
\]
and, in particular, $\big|\nabla f (x)\big|\le \gamma \Big(\big| \chi'\big( \rho d(x)\big)\big| e^{k \gamma \rho_L(x)} +k f(x)\Big)$.
Using Proposition~\ref{prop-ineq} we wave
\[
\big|\nabla f (x)\big|^p\le (1+\varepsilon_1)k^p \gamma^p f(x)^p + \dfrac{c \gamma^p}{\varepsilon_1^{p-1}}
\Big|\chi'\big( \gamma \rho(x)\big)\Big|^p e^{pk \gamma \rho_L(x)}, \quad \varepsilon_1\in(0,1),
\]
implying
\begin{align*}
\int_\Omega u^p |\nabla f|^p\dd x
&\le (1+\varepsilon_1)k^p \gamma^p
\int_\Omega u^p f^p \dd x
+\dfrac{c \gamma^p}{\varepsilon_1^{p-1}}\int_\Omega \Big|\chi'\big( \gamma \rho(x)\big)\Big|^p e^{pk \gamma \rho_L(x)} u(x)^p \dd x\\
& \le
(1+\varepsilon_1)k^p \gamma^p
\int_\Omega u^p f^p \dd x
+\dfrac{c\gamma^p}{\varepsilon_1^{p-1}} \int_{x\in\Omega: \rho(x)\le \frac{a}{\gamma}}
\Big|\chi'\big( \gamma \rho(x)\big)\Big|^p e^{pk \gamma \rho_L(x)} u(x)^p \dd x\\
&\le
(1+\varepsilon_1)k^p \gamma^p
\int_\Omega u^p f^p \dd x
+\dfrac{c c_0^p\gamma^p}{\varepsilon_1^{p-1}} e^{pka},
\end{align*}
where we used the normalization of $u$ on the last step.
The substitution into \eqref{eq-f1} gives
\begin{equation}
       \label{eq-f2}
\bigg((1-\varepsilon_0)^p
- \dfrac{1+\varepsilon_1}{(1-\delta_0)^{p-1}} p^p k^p
\bigg)\gamma^p\int_\Omega u^p f^p \dd x
+\delta_0\int_\Omega |\nabla u|^p f^p\dd x
\le 
\dfrac{C \gamma^p}{(1-\delta_0)^{p-1}\varepsilon_1^{p-1}}\,e^{pka}
\end{equation}
with $C:=c c_0^p p^p$. Note that all the estimates are uniform in the parameter
$L$ entering the definition of $\rho_L$, hence,
one can send $L$ to $+\infty$, which means that \eqref{eq-f2}
also holds for
\[
f(x):=\chi\big( \gamma \rho(x)\big) e^{k \gamma \rho(x)}.
\]

Recall that $\varepsilon_0\in [0,1)$ must satisfy \eqref{eq-eps0} while
$\delta_0 \in(0,1)$, $\varepsilon_1\in(0,1)$ and $k >0$ are arbitrary.
In particular, if $k\in(0,p^{-1})$ is fixed, then one can choose $\varepsilon_0$, $\varepsilon_1$
and $\delta_0$ positive but sufficiently small to have
\[
(1-\varepsilon_0)^p
- \dfrac{1+\varepsilon_1}{(1-\delta_0)^{p-1}} \ p^p k^p=:b>0
\]
implying
\[
b\gamma^p\int_\Omega u^p f^p \dd x
+\delta_0\int_\Omega |\nabla u|^p f^p\dd x\le 
C'\gamma^p, \quad C':=\dfrac{C}{(1-\delta_0)^{p-1}\varepsilon_1^{p-1}}e^{pka},
\]
and \eqref{eq-decay} follows from
\begin{align*}
&\int\limits_{\dist(x,\partial\Omega)> \frac{a}{\gamma}}
\Big(\big|\nabla u(x)\big|^p +\gamma^p \big|u(x)\big|^p\Big)
\exp\Big(\tau \gamma \dist(x,\partial\Omega)\Big)\dd x\\
&\qquad\qquad\qquad \le\int\limits_{\dist(x,\partial\Omega)> \frac{a}{\gamma}}
\Big(\big|\nabla u(x)\big|^p +\gamma^p \big|u(x)\big|^p\Big)
f(x)^p\dd x, \qquad \tau:= kp \in (0,1).
\end{align*}

If $\Omega$ is a bounded Lipschitz domain, then by Proposition~\ref{prop-lip1} the above constructions work
with $\gamma=(p-1)^\frac{1}{p}\alpha^\frac{1}{p-1}$ and $\varepsilon_0=0$,
then we can set $\delta_0=\varepsilon_1=\gamma^{-1}$ and $k=p^{-1}- B\gamma^{-1}$
with $B>0$ sufficiently large, which gives
\begin{equation}
      \label{eq-esta}
(1-\varepsilon_0)^p
- \dfrac{1+\varepsilon_1}{(1-\delta_0)^{p-1}} \ p^p k^p \ge \dfrac{1}{\gamma}, \quad \alpha\to + \infty,
\end{equation}
and one can proceed in the same way to obtain \eqref{eq-decay2}.

If $\Omega$ is an admissible domain, then by Theorem~\ref{thm-main}
we can take $\gamma=(p-1)^\frac{1}{p}\alpha^\frac{1}{p-1}$
and $\varepsilon_0:=A\gamma^{-1}$, with a suitable large $A>0$, then
by setting $\delta=\varepsilon_1=\gamma^{-1}$ and taking $k=p^{-1}- B \gamma^{-1}$
with a suitable large $B>0$ we obtain the estimate \eqref{eq-esta}
implying \eqref{eq-decay2} again.
\end{proof}

We mention a simple but important consequence which will be used below. Recall that $\Omega_\delta$ and $\Theta_\delta$
are defined in \eqref{eq-td}.
\begin{cor}\label{cor4}
Let $\Omega$ be an admissible domain such that the problem \eqref{eq-inf} admits a minimizer $u$, which
we assume to be non-negative and normalized by $\|u\|_{L^p(\Omega)}=1$,
then for any $\delta>0$ and any $N>0$ there holds
$\|u\|_{W^{1,p}(\Theta_\delta)}=o(\alpha^{-N})$ as $\alpha\to+\infty$.
\end{cor}

Finally we are in position to prove a weak form of a localization of the minimizer near the set
at which the mean curvature of the boundary takes its maximal value.

\begin{theorem}\label{thm-locbd}
Let $\Omega$ be an admissible domain such that the problem \eqref{eq-inf} admits a minimizer $u\equiv u_\alpha$
for large $\alpha$, which is assumed be chosen non-negative and normalized by $\|u\|_{L^p(\Omega)}=1$. Define
$\cH:\Omega\to \RR$ by $\cH(x)=H\big(s(x)\big)$, where
$s(x)\in \partial\Omega$ is given by $\dist(x,\partial\Omega)=\big|x-s(x)\big|$,
then
\begin{equation}
   \label{eq-lbd}
\int_{\Omega} \big(H_\mx - \cH)\, u^p \,\dd x=o(1), \quad \alpha\to+\infty.
\end{equation}
\end{theorem}  

\begin{proof}
It is well known that $s(x)$ is uniquely defined for almost all $x\in\Omega$.
Moreover, in view of Corollary~\ref{cor4}, it is sufficient to show that
\begin{equation}
   \label{eq-od}
\int_{\Omega_\delta} \big(H_\mx - \cH)\, u^p \dd x=o(1), \quad \alpha\to+\infty,
\end{equation}
for some $\delta>0$. We assume that $\delta$ is sufficiently small such that
the map \eqref{eq-st} is bijective, then for $x\in\Omega_\delta$
one has $s=s(x)$ iff $x=\Phi(s,t)$
for some $t\in(0,\delta)$. Furthermore, by the constructions of Section~\ref{sec2}
one has
\begin{align}
\Lambda(\Omega,p,\alpha) \|u\|^p_{L^p(\Omega)}
&=\int_\Omega |\nabla u|^p\dd x-\alpha \int_{\partial\Omega} u^p\,\dd\sigma  \ge
\int_{\Omega_\delta} |\nabla u|^p\dd x-\alpha \int_{\partial\Omega} u^p\,\dd\sigma \nonumber \\
&\ge \label{eq-lll}
\int_{\Sigma} \Big| C_- \big|\nabla_s v(s,t)\big|^2 + v_t(s,t)^2\Big|^{\frac p 2} \varphi(s,t) \dd\sigma(s)\dd t-\alpha \int_S \big|v(s,0)\big|^p\,\dd\sigma(s),
\end{align}
where $v:=u\circ \Phi$ and $C_->0$. Using Lemma~\ref{lemin} we have
\begin{align*}
& \int_{\Sigma} \Big| C_- \big|\nabla_s v(s,t)\big|^2 + v_t(s,t)^2\Big|^{\frac p 2} \varphi(s,t) \dd\sigma(s)\dd t-\alpha \int_S \big|v(s,0)\big|^p\,\dd\sigma(s)\\
& \qquad \ge
\int_{\Sigma} \big|v_t(s,t)\big|^p \varphi(s,t) \dd\sigma(s)\dd t-\alpha \int_S \big|v(s,0)\big|^p\,\dd\sigma(s)\\
& \qquad \ge \int_\Sigma  \Big[(1-p)\alpha^{\frac{p}{p-1}}-(\nu-1)H(s)\alpha+o(\alpha)\Big] v(s,t)^p
 \varphi(s,t) \dd\sigma(s)\dd t\\
& \qquad  =(1-p)\alpha^{\frac{p}{p-1}} \|u\|^p_{L^p(\Omega_\delta)}
-\alpha (\nu-1)\int_{\Omega_\delta} \cH u^p \dd x + o(\alpha) \|u\|^p_{L^p(\Omega_\delta)}.
\end{align*}
The substitution into \eqref{eq-lll} and the asymptotic expansion \eqref{eq-main} for $\Lambda(\Omega,p,\alpha)$ give
\begin{multline*}
(1-p)\alpha^{\frac{p}{p-1}} \|u\|^p_{L^p(\Omega)}-\alpha (\nu-1) H_\mx \|u\|^p_{L^p(\Omega)}
+o(\alpha) \|u\|^p_{L^p(\Omega)}\\
\ge
(1-p)\alpha^{\frac{p}{p-1}} \|u\|^p_{L^p(\Omega_\delta)}
-\alpha(\nu-1) \int_{\Omega_\delta} \cH u^p \dd x + o(\alpha) \|u\|^p_{L^p(\Omega_\delta)}.
\end{multline*}
Using $1=\|u\|^p_{L^p(\Omega)}=\|u\|^p_{L^p(\Omega_\delta)}+\|u\|^p_{L^p(\Theta_\delta)}$ we arrive at
\[
\alpha \int_{\Omega_\delta} (H_\mx -\cH)\, u^p \,\dd x \le
o(\alpha)-\dfrac{p-1}{\nu-1}\,\alpha^{\frac{1}{p-1}} \|u\|^p_{L^p(\Theta_\delta)}
-H_\mx \|u\|^p_{L^p(\Theta_\delta)},
\]
and the result follows from Corollary~\ref{cor4}.
\end{proof}

\appendix

\section{\bf Solvable cases}\label{seca}

For the sake of completeness, let us mention some cases in which  $\Lambda(\Omega,p,\alpha)$
can be computed explicitly.

\begin{prop}
For any $p\in(1,\infty)$ and $\alpha>0$ there holds $\Lambda(\RR_+,p,\alpha)=(1-p)\alpha^\frac{p}{p-1}$, and
the minimizer $u_*$ for Eq.~\eqref{eq-inf} is  given by $u_*(t)=\exp \big(-\alpha^{\frac{1}{p-1}} t\big)$.
\end{prop}

\begin{proof}
By computing the right-hand side of \eqref{eq-inf}  for $u=u_*$ we obtain the inequality
$\Lambda(\RR_+,p,\alpha)\le (1-p)\alpha^\frac{p}{p-1}$.
For the reverse inequality, we remark that  $\lim_{x\to+\infty} u(x)=0$
for any $u\in W^{1,p}(\RR_+)$ and, using the H\"older inequality,
\begin{gather*}
\big|u(0)\big|^p
=
-p\int_0^{+\infty} \big| u\big|^{p-1}
| u|' \,\dd t
\le
p\int_0^{+\infty} \big| u\big|^{p-1}
\big|| u|'\big| \,\dd t\le
p\|u\|_p^{p-1} \big\| |u|' \big\|_p
\le
p\|u\|^{p-1}_p \|u'\|_p.
\end{gather*}
Therefore,
\begin{multline*}
\inf_{\substack{u\in W^{1,p}(\RR_+)\\ u\not\equiv 0}}\dfrac{\displaystyle \int_0^{+\infty} \big|u'(t)\big|^p \dd t - \alpha\big|u(0)\big|^p}{\displaystyle\int_0^{+\infty} |u(t)|^p\dd t}
=
\inf_{\substack{v\in W^{1,p}(\RR_+)\\ \|v\|_p=1}}
\Big( \|v'\|^p_p - \alpha\big|v(0)\big|^p\Big)\\
\ge \inf_{\substack{v\in W^{1,p}(\RR_+)\\ \|v\|_p=1}}
\Big( \|v'\|^p_p - \alpha p \|v'\|_p\Big)
\ge \inf_{x\in \RR_+} (x^p-p\alpha x)=(1-p)\alpha^\frac{p}{p-1},
\end{multline*}
which gives the sought result.
\end{proof}

As observed in \cite{lp}, the one-dimensional result can be used to study
the  infinite planar sectors
\[
U_\theta:= \big\{ (x_1,x_2)\in\RR^2: \big|\arg (x_1+i x_2)\big|<\theta\big\}, \quad 0<\theta<\pi.
\]
Proceeding literally as in Lemma 2.6 and Lemma 2.8 of \cite{lp} one arrives at the following result:
\begin{prop}\label{propa2}
Let $\alpha>0$ and $p\in(1,+\infty)$, then for $\theta\ge \dfrac{\pi}{2}$ there holds $\Lambda(U_\theta,p,\alpha)=(1-p)\alpha^\frac{p}{p-1}$, while
for $\theta <\dfrac{\pi}{2}$ one has
\[
\Lambda(U_\theta,p,\alpha)=(1-p)\Big(\dfrac{\alpha}{\sin\theta}\Big)^\frac{p}{p-1}<(1-p)\alpha^\frac{p}{p-1},
\] 
which is attained on
\[
u(x_1,x_2)=\exp\Big(- \Big(\dfrac{\alpha}{\sin\theta}\Big)^{\frac{1}{p-1}}x_1\Big).
\]

\end{prop}

\section{\bf An auxiliary inequality}\label{secb}

\begin{prop}\label{prop-ineq}
Let $p>1$, then for any $\varepsilon\in(0,1)$
and for all $a,b\ge 0$ there holds
\[
(a+b)^p\le (1+\varepsilon) a^p +\dfrac{c}{\varepsilon^{p-1}}\, b^p,
\quad
c:= \max\Big\{\big(1-2^{\frac{1}{1-p}}\big)^{1-p},1\Big\}.
\]
\end{prop}

\begin{proof}
By homogenity, it is sufficient to show that $(1+t)^p\le (1+\varepsilon) t^p+ c\varepsilon^{1-p}$ for
all $t\ge 0$. Denote
$f_\varepsilon(t):=(1+t)^p-(1+\varepsilon) t^p$, then one simply needs an upper estimate for
$C(\varepsilon):=\varepsilon^{p-1}\sup_{t\in\RR_+}f_\varepsilon(t)$. We have $f_\varepsilon(0)=1$ and $f_\varepsilon(+\infty)=-\infty$.
The equation $f'_\varepsilon(t)=0$ has a unique solution
\[
t=t_\varepsilon=\dfrac{1}{(1+\varepsilon)^{\frac{1}{p-1}}-1},
\quad
f(t_\varepsilon)=\dfrac{1+\varepsilon}{\big((1+\varepsilon)^{\frac{1}{p-1}}-1\big)^{p-1}},
\]
implying
\[
C(\varepsilon)=\max\Big\{\dfrac{\varepsilon^{p-1}(1+\varepsilon)}{\big((1+\varepsilon)^{\frac{1}{p-1}}-1\big)^{p-1}},1\Big\}.
\]
Denote $\sigma:=(1+\varepsilon)^{\frac{1}{p-1}}-1$, then
\[
\sup_{\varepsilon\in(0,1)} \dfrac{\varepsilon^{p-1}(1+\varepsilon)}{\big((1+\varepsilon)^{\frac{1}{p-1}}-1\big)^{p-1}}
=\sup_{\sigma\in (0,\sigma_1)}\Big(\dfrac{(1+\sigma)^p-(1+\sigma)}{\sigma}\Big)^{p-1}, \quad
\sigma_1:= 2^{\frac{1}{p-1}}-1.
\]
Using the convexity of $\varphi(\sigma)=(1+\sigma)^p-(1+\sigma)$ we have
\[
\varphi(\sigma)\le \varphi(0) + \dfrac{\varphi(\sigma_1)-\varphi(0)}{\sigma_1}\, \sigma\equiv
\dfrac{2^{\frac{1}{p-1}}}{2^{\frac{1}{p-1}}-1}\, \sigma, \quad \sigma\in(0,\sigma_1),
\]
which gives the sought inequality.
\end{proof}

\section{\bf Remainder estimates for more regular domains}\label{appc}

If a stronger regularity of $\Omega$ is imposed, the remainder in Theorem~\ref{thm-main}
can be made more explicit.

\begin{prop}\label{pc1}
In Theorem~\ref{thm-main} assume additionally that the boundary of $\Omega$ it $C^3$ smooth
and that the mean curvature attains the maximum value $H_\mx$, then the remainder estimate in \eqref{eq-main}
can
be improved to $\cO(\alpha^{1-\kappa})$ with
\begin{equation}
  \label{eq-k1}
\kappa=\begin{cases}
\dfrac{1}{p+1}, & p\in (1,2],\\[\bigskipamount]
\dfrac{1}{3(p-1)}, & p \in (2,\infty).
\end{cases}
\end{equation}
If, in addition, $\Omega$ is $C^4$ smooth, then one can take
\begin{equation}
  \label{eq-k2}
\kappa=\begin{cases}
\dfrac{2}{p+2}, & p\in (1,2],\\[\bigskipamount]
\dfrac{1}{2(p-1)}, & p \in (2,\infty).
\end{cases}
\end{equation}
\end{prop}

\begin{proof}
Remark first that the result of section~\ref{sec4} imply
\[
\Lambda(\Omega,p,\alpha)\ge(1-p)\alpha^{\frac{p}{p-1}}-(\nu-1) H_\mx\alpha+\cO(\alpha^{\frac{p}{p-1}}\log \alpha),
\]
and $\alpha^{\frac{p}{p-1}}\log \alpha=o(\alpha^{1-\kappa})$ for $\kappa$ given by \eqref{eq-k1} or \eqref{eq-k2}.
Therefore, it is sufficient to show the upper bound, which
will be done by taking another test function in the computations of Section~\ref{sec3}.

Let $s_0\in S$ be such that $M(s_0)=M_\mx$. As $\partial\Omega$ is $C^3$ smooth, then $M$ is  $C^1$,
and  for some $m\ge 0$ we have $-M_\mx\le -M(s)\le -M_\mx + m d(s,s_0)$ with $d(\cdot,\cdot)$ standing
for the geodesic distance on $S$ and for $s$ sufficiently close to $s_0$,

Let us choose a function $f\in C_c^\infty(\RR)$ which equals $1$ in a neighborhood
of the origin and consider the functions $v\in W^{1,p}(S)$ given by
\begin{equation}
       \label{eq-vf}
v(s)= f\Big(\dfrac{d(s,s_0)}{\mu}\Big),
\end{equation}
where $\mu$ is a positive parameter which tends to $0$ as $\beta\to+\infty$
and will be chosen later.
One can estimate, for large $\beta$,
\begin{equation}
   \label{eq-fint}
\begin{aligned}
A:=\int_S \big|v(s)\big|^p \dd\sigma(s)&=a \mu^{\nu-1}+\cO(\mu^\nu),\quad A^{-1}=\cO(\mu^{1-\nu}),\\
\int_S \big(-M(s)\big) \big|v(s)\big|^p \dd\sigma(s)&= -M_\mx A + \cO(\mu^\nu),\\
\int_S \big|\nabla v(s)\big|^p\dd\sigma(s)&=\cO(\mu^{\nu-p-1}).
\end{aligned}
\end{equation}

Consider first the case $p\in(1,2]$. The substitution into \eqref{eq-lplus} gives 
\begin{align*}
\Lambda^+(p,\alpha)&\le \Bigg\{
\cO(\mu^{\nu-p-1}\beta^{-1}) + \dfrac{1}{p}\beta^{p-1} A - \dfrac{1}{p^2} \beta^{p-2}M_\mx A+\cO(\mu^\nu \beta^{p-2})+\cO(\mu^{\nu-1}\beta^{p-3}) - \beta^{p-1} A\Bigg\}\\
&\quad \times \Big\{
\dfrac{1}{p\beta} A - \dfrac{M_\mx}{p^2\beta^2} A + \cO(\mu^\nu \beta^{-2})+\cO(\mu^{\nu-1}\beta^{-3})
\Big\}^{-1}\\
&=\Big\{
\dfrac{1-p}{p} \beta^{p-1} A - \dfrac{1}{p^2} M_\mx \beta^{p-2} A + \cO(\mu^{\nu-p-1}\beta^{-1}+\mu^\nu \beta^{p-2}+\mu^{\nu-1}\beta^{p-3})
\Big\}\\
&\quad \times \Big\{ \dfrac{A}{p\beta} \Big(1 - \dfrac{M_\mx}{p\beta}+ \cO(\mu \beta^{-1}+\beta^{-2})
\Big)
\Big\}^{-1}\\
&=\Big(
(1-p)\beta^p - \dfrac{1}{p} M _\mx \beta^{p-1}+ \cO\big(\mu^{-p}+\mu \beta^{p-1}+\beta^{p-2}\big)
\Big)  \Big(1 + \dfrac{M_\mx}{p\beta}+ \cO(\mu \beta^{-1}+\beta^{-2})
\Big)\\
&=(1-p)\beta^p - M_\mx \beta^{p-1} +\cO\big( \mu^{-p}+\mu \beta^{p-1}+\beta^{p-2}\big).
\end{align*}
The remainder is optimized by $\mu=\beta^{-\frac{p-1}{p+1}}$, and we arrive
\begin{align*}
\Lambda^+(p,\alpha)&\le (1-p)\beta^p - M_\mx \beta^{p-1} + \cO(\beta^{p-1 - \frac{p-1}{p+1}})\\
&= (1-p)\alpha^\frac{p}{p-1} - M_\mx \alpha + \cO\big(\alpha^{1- \frac{1}{p+1}}\big).
\end{align*}

If, in addition, $\Omega$ is $C^4$ smooth, then $M$ is $C^2$ smooth, 
and for some $m\ge 0$ we have
$-M_\mx\le -M(s)\le -M_\mx + m d(s,s_0)^2$ as $s$ is sufficiently close to $s_0$, which allows one to replace
the second estimate in \eqref{eq-fint} by
\begin{equation}
         \label{eq-c4}
\int_S \big(-M(s)\big) \big|v(s)\big|^p \dd\sigma(s)= -M_\mx A + \cO(\mu^{\nu+1}),
\end{equation}
and a similar computation gives
\[
\Lambda^+(p,\alpha)\le (1-p)\beta^p - M_\mx \beta^{p-1} +\cO\big( \mu^{-p}+\mu^2 \beta^{p-1}+\beta^{p-2}\big).
\]
Hence, taking $\mu=\beta^{-\frac{p-1}{p+2}}$ we arrive at
\begin{align*}
\Lambda^+(p,\alpha)&\le (1-p)\beta^p - M_\mx \beta^{p-1} + \cO(\beta^{p-1 - \frac{2(p-1)}{p+2}})\\
&= (1-p)\alpha^\frac{p}{p-1} - M_\mx \alpha + \cO\big(\alpha^{1- \frac{2}{p+2}}\big).
\end{align*}

For the case $p\in(2,+\infty)$, the substitution of \eqref{eq-fint} into \eqref{eq-lam2}
gives
\begin{align*}
\Lambda^+(p,\alpha)&\le \Big\{ \dfrac{1-p}{p} \beta^{p-1} A - \dfrac{M_\mx}{p^2}  \beta^{p-2} A + \cO\big( \varepsilon_0^\frac{2-p}{2}\mu^{\nu-p-1}\beta^{-1} + \mu^\nu \beta^{p-2}\\
&\quad{}+ \mu^{\nu-1}\beta^{p-3}
+\varepsilon_0 \mu^{\nu-1}\beta^{p-1}
\big)\Big\} \times \Big\{ \dfrac{A}{p\beta} \Big(1 - \dfrac{M_\mx}{p\beta}+ \cO(\mu \beta^{-1}+\beta^{-2})
\Big)
\Big\}^{-1}\\
&=\Big(
(1-p)\beta^p - \dfrac{1}{p} M _\mx \beta^{p-1}+ \cO\big(\varepsilon_0^\frac{2-p}{2}\mu^{-p}+\mu \beta^{p-1}+\beta^{p-2}
+\varepsilon_0 \beta^p\big)
\Big)\\
&\quad \times \Big(1 + \dfrac{M_\mx}{p\beta}+ \cO(\mu \beta^{-1}+\beta^{-2})
\Big)\\
&=(1-p)\beta^p -  M _\mx \beta^{p-1}+ \cO\big(\varepsilon_0^\frac{2-p}{2}\mu^{-p}+\mu \beta^{p-1}+\beta^{p-2}
+\varepsilon_0 \beta^p\big).
\end{align*}
In order to optimize we solve
$\varepsilon_0^\frac{2-p}{2}\mu^{-p}=\mu \beta^{p-1}=\varepsilon_0 \beta^p$,
which gives
\[
\mu=\beta^{-1/3}, \quad \varepsilon_0=\beta^{-4/3}, \quad
\cO\big(\varepsilon_0^\frac{2-p}{2}\mu^{-p}+\mu \beta^{p-1}+\beta^{p-2}
+\varepsilon_0 \beta^p\big)=\cO(\beta^{p-\frac{4}{3}}),
\]
and, therefore, 
\[
\Lambda^+(p,\alpha)\le (1-p)\beta^p - M_\mx \beta^{p-1} + \cO(\beta^{p-\frac{4}{3}})
= (1-p)\alpha^\frac{p}{p-1} - M_\mx \alpha + \cO\big(\alpha^{1- \frac{1}{3(p-1)}}\big).
\]

For $C^4$ domains, a similar computation using \eqref{eq-c4} gives
\[
\Lambda^+(p,\alpha)\le 
(1-p)\beta^p -  M _\mx \beta^{p-1}+ \cO\big(\varepsilon_0^\frac{2-p}{2}\mu^{-p}+\mu^2 \beta^{p-1}+\beta^{p-2}
+\varepsilon_0 \beta^p\big).
\]
The remainder is optimized by
\[
\mu=\beta^{-1/4}, \quad \varepsilon_0=\beta^{-3/2}, \quad
\cO\big(\varepsilon_0^\frac{2-p}{2}\mu^{-p}+\mu \beta^{p-1}+\beta^{p-2}
+\varepsilon_0 \beta^p\big)=\cO(\beta^{p-\frac{3}{2}}),
\]
which gives the sought result. We remark that for $p=2$ in \eqref{eq-k2} we obtain $\kappa=1/2$, which is optimal, see \cite{HK,pp15b}.
\end{proof}

Finally, an easy  revision of the proof of Theorem~\ref{thm-locbd} gives the following result:
\begin{cor}
Under the assumptions of Proposition~\ref{pc1}, the right-hand-side of \eqref{eq-lbd}
can be improved to $\cO(\alpha^{-\kappa})$.
\end{cor}

\section*{\bf Acknowledgments}  H.~K.  has been partially supported by Gruppo Nazionale per Analisi Matematica, la Probabilit\`a e le loro Applicazioni (GNAMPA) of the Istituto Nazionale di Alta Matematica (INdAM). The support of  MIUR-PRIN2010-11 grant for the project ``Calcolo delle variazioni'' (H.~K.), is also gratefully acknowledged. K.P. has been partially supported by CNRS GDR 2279 DynQua. The authors thank Carlo Nitsch and Cristina Trombetti for useful comments on a preliminary version of the work.

\end{document}